\documentclass[10pt]{amsart}
\usepackage{amsfonts}
\usepackage[foot]{amsaddr}
\usepackage{mathrsfs,amssymb,amsmath}
\usepackage{verbatim}
\usepackage{hyperref}
\usepackage{graphicx}
\usepackage{enumerate}
\usepackage{vmargin}
\usepackage{caption}
\usepackage{subcaption}
\usepackage{layout}
\usepackage{bigints}
\usepackage{booktabs}
\usepackage{color}

\allowdisplaybreaks
\newtheorem{theorem}{Theorem}[]

\newtheorem{proposition}[theorem]{Proposition}
\newtheorem{lemma}[theorem]{Lemma}
\newtheorem{corollary}[theorem]{Corollary}

\theoremstyle{definition}

\newtheorem{remark}{Remark}[section]
\numberwithin{equation}{section}
\input{tcilatex}

\begin{document}
\title[Critical Values Random Spherical Harmonics]{A Reduction Principle for
the Critical Values of Random Spherical Harmonics}
\author{Valentina Cammarota}
\address{Department of Statistics, Sapienza University of Rome}
\author{Domenico Marinucci}
\address{Department of Mathematics, University of Rome Tor Vergata}
\maketitle

\begin{abstract}
We study here the random fluctuations in the number of critical points with values in an interval $%
I\subset \mathbb{R}$ for Gaussian spherical eigenfunctions $\left\{
f_{\ell }\right\} $, in the high energy regime
where $\ell \rightarrow \infty $. We show that these fluctuations are asymptotically equivalent to the centred $L^{2}$-norm of $\left\{ f_{\ell }\right\} $ times the integral of a (simple and fully
explicit) function over the interval under consideration. We discuss also
the relationships between these results and the asymptotic behaviour of other geometric functionals on
the excursion sets of random spherical harmonics.
\end{abstract}

\date{\today }

\begin{itemize}
\item \textbf{AMS Classification}: 60G60, 62M15, 53C65, 42C10, 33C55.

\item \textbf{Keywords and Phrases}: {Reduction Principle, Critical Points,
Wiener-Chaos Expansion, Spherical Harmonics, Quantitative Central Limit
Theorem, Berry's Cancellation Phenomenon}
\end{itemize}

\section{Introduction and Main Result}

\subsection{The asymptotic geometry of random spherical harmonics}

It is well-known that the eigenvalues $\lambda $ of the Laplace equation $%
\Delta _{\mathbb{S}^{2}}f+\lambda f=0$ on the two-dimensional sphere $%
\mathbb{S}^{2}$, are of the form $\lambda =\lambda _{\ell }=\ell (\ell +1)$
for some integer $\ell \geq 1$. For any given eigenvalue $\lambda _{\ell }$,
the corresponding eigenspace is the $(2\ell +1)$-dimensional space of
spherical harmonics of degree $\ell $; we can choose an arbitrary $L^{2}$%
-orthonormal basis $\left\{ Y_{\mathbb{\ell }m}(.)\right\} _{m=-\ell ,\dots
,\ell }$, and consider random eigenfunctions of the form
\begin{equation*}
f_{\ell }(x)=\frac{\sqrt{4\pi }}{\sqrt{2\ell +1}}\sum_{m=-\ell }^{\ell
}a_{\ell m}Y_{\ell m}(x),
\end{equation*}%
where the coefficients $\left\{ a_{\mathbb{\ell }m}\right\} $ are
complex-valued Gaussian variables, such that for $m\neq 0$, $\text{Re}%
(a_{\ell m})$, $\text{Im}(a_{\ell m})$ are zero-mean, independent Gaussian
variables with variance $\frac{1}{2}$, while $a_{\ell 0}$ follows a standard
Gaussian distribution. The random fields $\{f_{\ell }(x),\;x\in \mathbb{S}%
^{2}\}$ are isotropic, meaning that the probability laws of $f_{\ell }(\cdot
)$ and $f_{\ell }^{g}(\cdot ):=f_{\ell }(g\cdot )$ are the same for any
rotation $g\in SO(3)$. Also, $f_{\ell }$ are centred Gaussian, and from the
addition theorem for spherical harmonics (see i.e., \cite{MaPeCUP}, eq.
(3.42)) the covariance function is given by,
\begin{equation}
\mathbb{E}[f_{\ell }(x)f_{\ell }(y)]=P_{\ell }(\cos \theta _{xy})\text{ },%
\hspace{0.5cm}\theta _{xy}:=d_{\mathbb{S}^{2}}(x,y)\text{ },  \label{cov}
\end{equation}%
where $P_{\ell }$ are the usual Legendre polynomials, defined by%
\begin{equation*}
P_{\ell }(t)=\frac{1}{2^{\ell }\ell !}\frac{d^{\ell }}{dt^{\ell }}%
(t^{2}-1)^{\ell }\text{ },\hspace{0.5cm}t\in \lbrack -1,1],\hspace{0.5cm}%
\ell \in \mathbb{N}\text{ },
\end{equation*}%
whereas $d_{\mathbb{S}^{2}}(x,y)$ is the usual spherical geodesic distance
between $x$ and $y,$ i.e.,%
\begin{equation*}
d_{\mathbb{S}^{2}}(x,y)=\arccos (\left\langle x,y\right\rangle ),\text{
where }\left\langle x,y\right\rangle =\cos \theta _{x}\cos \theta _{y}+\sin
\theta _{x}\sin \theta _{y}\cos (\varphi _{x}-\varphi _{y})\text{ },
\end{equation*}%
and $(\theta _{x},\varphi _{x})$, $(\theta _{y},\varphi _{y})$ are the
spherical coordinates of $x$ and $y$ respectively, $\theta _{a}\in \lbrack
0,\pi ]$, $\varphi _{a}\in \lbrack 0,2\pi ),$ $a=x,y.$

A number of issues on the geometry of random spherical harmonics has been
scrutinized recently, including the number of nodal domains \cite{nazarov},
the length of nodal lines \cite{Wig}, \cite{MRW}, the excursion area and the
defect \cite{MW}, \cite{MR2015}, critical values and the Euler-Poincar\'{e}
characteristic \cite{CMW}, \cite{CMW-EPC}, \cite{CW}, \cite{CM2018}, mass
equidistribution \cite{han}, critical radius \cite{feng}; these and/or other
geometric features for random eigenfunctions on other compact manifolds such
as the torus (arithmetic random waves) or on the plane (Berry's random waves
model, \cite{Berry 1977}) have also been intensively studied, see i.e. \cite%
{BMW}, \cite{dalmao}, \cite{KKW}, \cite{MPRW2015}, \cite{npr}, \cite%
{peccatirossi} for the fluctuations of nodal lengths, \cite{Buckley} for the
number of nodal domains, \cite{granville} for the analysis of mass
equidistributions and \cite{Rudnick}, \cite{RudnickYesha} for nodal
intersections, to list some of the most recent contributions; a review is
given in \cite{Rossi2018}.

In particular, consider the excursion sets $A_{u}(f_{\ell };\mathbb{S}^{2}),$
defined for $u\in \mathbb{R}$ by
\begin{equation*}
A_{u}(f_{\ell };\mathbb{S}^{2}):=\left\{ x\in \mathbb{S}^{2}:f_{\ell
}(x)\geq u\right\} \text{ };
\end{equation*}%
their geometry can be characterized by the behaviour of the
Lipschitz-Killing curvatures $\mathcal{L}_{k},$ $k=0,1,2,$ (see i.e., \cite%
{adlertaylor},\cite{adlerstflour}), which in dimension two correspond
respectively to the Euler-Poincar\'{e} characteristic $\mathcal{L}_{0},$
(half) the boundary length $\mathcal{L}_{1}$, and the excursion area $%
\mathcal{L}_{2};$ the asymptotic behaviour (in the high-energy regime $\ell
\rightarrow \infty $) of these three quantities is now fully understood, and
indeed we have (see \cite{CW})%
\begin{equation}
\mathcal{L}_{0}(A_{u}(f_{\ell };\mathbb{S}^{2}))-\mathbb{E}\left[ \mathcal{L}%
_{0}(A_{u}(f_{\ell };\mathbb{S}^{2}))\right] =\frac{1}{2}\left\{ \frac{%
\lambda _{\ell }}{2}\right\} \left[ H_{2}(u)H_{1}(u)\phi (u)\right] \frac{1}{%
2\pi }\int_{\mathbb{S}^{2}}H_{2}(f_{\ell }(x))dx+R_{0}(\ell )\text{ },
\label{aop1}
\end{equation}%
\begin{equation}
\mathcal{L}_{1}(A_{u}(f_{\ell };\mathbb{S}^{2}))-\mathbb{E}\left[ \mathcal{L}%
_{1}(A_{u}(f_{\ell };\mathbb{S}^{2}))\right] =\frac{1}{2}\left\{ \frac{%
\lambda _{\ell }}{2}\right\} ^{1/2}\sqrt{\frac{\pi }{8}}\left[
H_{1}^{2}(u)\phi (u)\right] \int_{\mathbb{S}^{2}}H_{2}(f_{\ell
}(x))dx+R_{1}(\ell )\text{ },  \label{aop2}
\end{equation}%
\begin{equation}
\mathcal{L}_{2}(A_{u}(f_{\ell };\mathbb{S}^{2}))-\mathbb{E}\left[ \mathcal{L}%
_{2}(A_{u}(f_{\ell };\mathbb{S}^{2}))\right] =\frac{1}{2}\left\{ \frac{%
\lambda _{\ell }}{2}\right\} ^{0}\left[ H_{0}(u)H_{1}(u)\phi (u)\right]
\int_{\mathbb{S}^{2}}H_{2}(f_{\ell }(x))dx+R_{2}(\ell )\text{ },
\label{aop3}
\end{equation}%
where $\phi (u)={1}/\sqrt{2\pi }\exp \{-{u^{2}}/{2}\}$ denotes a standard
Gaussian density, $H_{k}(.)$ denotes Hermite polynomials $%
H_{k}(u)=(-1)^{k}\phi (u)\frac{d^{k}\phi (u)}{du^{k}}$, and, for $k=0,1,2$, $%
R_{k}(\ell )$ are asymptotically negligible, meaning that
\begin{equation*}
\lim_{\ell \rightarrow \infty }\frac{\mathbb{E}[R_{k}^{2}(\ell )]}{\text{Var}%
\left( \mathcal{L}_{k}(A_{u}(f_{\ell };\mathbb{S}^{2}))\right) }=0\text{ }.
\end{equation*}%
The results given in \eqref{aop1}-\eqref{aop3} can be viewed as broadly
analogues to the \emph{reduction principles} established by \cite{DehTaq}
for the empirical processes of long-range dependent, stationary stochastic
processes on $\mathbb{Z}$: indeed these authors had shown that the empirical
process of long range dependence sequences is asymptotically fully
degenerate, being equivalent to a single random variable belonging to a
Wiener chaos rescaled by a deterministic function depending only on the
threshold value $u$. Likewise, the results reported in \eqref{aop1}-%
\eqref{aop3} entail some unexpected characterizations on the asymptotic
behaviour of these geometric functionals, namely

\begin{enumerate}
\item Any pair of them is asymptotically perfectly correlated at any
non-zero threshold level, i.e.
\begin{equation*}
\lim_{\ell \rightarrow \infty }\text{Corr}\left\{ \mathcal{L}%
_{j}(A_{u_{1}}(f_{\ell };\mathbb{S}^{2})),\mathcal{L}_{k}(A_{u_{2}}(f_{\ell
};\mathbb{S}^{2}))\right\} =1,\text{for all }j,k=0,1,2,u_{1},u_{2}\in
\mathbb{R},u_{1},u_{2}\neq 0\text{ };
\end{equation*}%
in words, knowledge of the value of any of the three functionals allows to
compute with asymptotically perfect precision the value of the other two
functionals at any (non-zero) level

\item It is possible to establish rather easily quantitative central limit
theorems for each of them, by using standard tools; note indeed that, by
Parseval's equality%
\begin{equation}
\int_{\mathbb{S}^{2}}H_{2}(f_{\ell }(x))dx=\frac{4\pi }{2\ell +1}%
\sum_{m=-\ell }^{\ell }\left\{ |a_{\ell m}|^{2}-\mathbb{E}|a_{\ell
m}|^{2}\right\} \text{ },  \label{H2}
\end{equation}%
i.e., each of the Lipschitz-Killing curvatures is a deterministic rescaling
of a sum of centred, i.i.d. variables

\item A phase transition occurs for $u=0,$ where the leading terms disappear
for all three functionals.
\end{enumerate}

\subsection{Statement of the main result}

Our purpose in this paper is to provide a similar characterization for the
behaviour of Gaussian spherical harmonics in terms of critical points. More
precisely, let $I\subseteq \mathbb{R}$ be any interval in the real line; we
are interested in the number of critical points with value in $I$:
\begin{equation*}
\mathcal{N}_{\ell }^{c}(I)=\#\{x\in \mathbb{S}^{2}:f_{\ell }(x)\in I,\nabla
f_{\ell }(x)=0\}\text{ }.
\end{equation*}%
It has been shown in \cite{CMW} that for every interval $I\subseteq \mathbb{R%
}$, as $\ell \rightarrow \infty $
\begin{equation*}
\mathbb{E}[\mathcal{N}_{\ell }^{c}(I)]=\frac{2}{\sqrt{3}}\ell (\ell
+1)\int_{I}\pi ^{c}(t)dt+O(1)\text{ },
\end{equation*}%
where we have introduced the density function.
\begin{equation*}
\pi ^{c}(t)=\frac{\sqrt{3}}{\sqrt{8\pi }}(2e^{-t^{2}}+t^{2}-1)e^{-\frac{t^{2}%
}{2}}\text{ }.
\end{equation*}%
The constant in the $O(\cdot )$ term is universal; here (and later) by
universality of the constant we mean that the integral of the error term on
any interval $I$ is uniformly bounded by its value when $I=\mathbb{R}$,
which is O(1). Also, we have
\begin{equation*}
\mathbb{E}[\mathcal{N}_{\ell }^{c}]:=\mathbb{E}[\mathcal{N}_{\ell }^{c}(%
\mathbb{R})]=\frac{2}{\sqrt{3}}\ell (\ell +1)+O(1)\text{ }.
\end{equation*}%
The investigation of the asymptotic variance of critical values is more
challenging; we need first to establish some more notation, i.e. we shall
define (see also \cite{CW})
\begin{equation}
p_{3}^{c}(t)=\frac{1}{\sqrt{8\pi }}e^{-\frac{3}{2}t^{2}}\left[
2-6t^{2}-e^{t^{2}}(1-4t^{2}+t^{4})\right] \text{ },\hspace{0.5cm}\nu ^{c}(I)=%
\left[ \int_{I}p_{3}^{c}(t)dt\right] ^{2}\text{ }.  \label{p3}
\end{equation}
The following result was given in \cite{CMW}: for every interval $I\subseteq
\mathbb{R}$ as $\ell \rightarrow \infty $
\begin{equation*}
{\text{Var}}(\mathcal{N}_{\ell }^{c}(I))=\ell ^{3}\nu ^{c}(I)+O(\ell ^{5/2})%
\text{ },
\end{equation*}%
again with a universal error bound in the $O(\cdot )$ term; similar results
hold for the number of extrema and saddles.

\begin{remark}
It can be noted that for some intervals $I$ the leading constant $\nu
^{c}(I) $ vanishes, and, accordingly, the order of magnitude of the variance
is smaller than $\ell ^{3}$; the most important among these cases is for $I=%
\mathbb{R}$ (corresponding to the total number of critical points), where we
have \cite{CW}, as $\ell \rightarrow \infty $
\begin{equation*}
\mathrm{Var}({\mathcal{N}}_{\ell }^{c})=\frac{1}{3^{3}\pi ^{2}}\ell ^{2}\log
\ell +O(\ell ^{2})\text{ }.
\end{equation*}%
More generally, for the intervals such that the constant $\nu ^{c}(I)$
vanishes, the variance of the number of critical points in $I$ has the
following asymptotic behaviour: as $\ell \rightarrow \infty $
\begin{equation}
\mathrm{Var}({\mathcal{N}}_{\ell }^{c}(I))=[\mu ^{c}(I)]^{2}\ell ^{2}\log
\ell +O(\ell ^{2})\text{ },  \label{fourthchaos}
\end{equation}%
where $\mu ^{c}(I)=\int_{I}\mu ^{c}(t)dt$, and the function $\mu ^{c}(t)$ is
given by
\begin{equation*}
\mu ^{c}(t)=\frac{1}{2^{3}\pi }\sqrt{\frac{2}{\pi }}%
[(-2-36t^{2}+38t^{4})e^{-t^{2}}+1+17t^{2}-11t^{4}+t^{6}]e^{-\frac{t^{2}}{2}}%
\text{ }.
\end{equation*}
\end{remark}

Our main result in this paper is to establish a reduction principle for the
number of critical points in the interval $I,$ of similar nature as those
given earlier in \eqref{aop1}-\eqref{aop3}. In particular, following the
same approach as given for other geometric functionals in many recent papers
(see i.e., \cite{MPRW2015}, \cite{CM2018}, \cite{npr} and the references
therein) we shall start by computing the $L^{2}(\Omega )$ expansion of
critical points into \emph{Wiener Chaoses} (the orthogonal spaces spanned by
Hermite polynomials, see \cite{noupebook} for details), which will lead to
lead to the $L^{2}(\Omega )$ expansion%
\begin{equation}
\mathcal{N}_{\ell }^{c}(I)=\sum_{q=0}^{\infty }\mathcal{N}_{\ell ;I}^{c}[q]%
\text{ ;}  \label{L2exp}
\end{equation}%
indeed the rigorous justification of (\ref{L2exp}) and is one of the main
technical challenges of this paper. The other main step is then to show that
a single term dominates the $L^{2}(\Omega )$ expansion (after centering),
then leading to the following

\begin{theorem}
\label{th1} As $\ell \rightarrow \infty $, we have that
\begin{align}
\mathcal{N}_{\ell }^{c}(I)-\mathbb{E}[\mathcal{N}_{\ell }^{c}(I)]& =\frac{%
\ell }{2}\,\left[ \int_{I}p_{3}^{c}(t)dt\right] \mathcal{N}_{\ell
;I}^{c}[2]+R_{\ell }(I)  \notag \\
& =\frac{\lambda _{\ell }}{2}\,\left[ \int_{I}p_{3}^{c}(t)dt\right] \frac{1}{%
2\pi }\int_{\mathbb{S}^{2}}H_{2}(f_{\ell }(x))dx+R_{\ell }(I) \\
& =\frac{\ell +1}{2}\,\left[ \int_{I}p_{3}^{c}(t)dt\right] \frac{2\ell }{%
2\ell +1}\;\sum_{m=-\ell }^{\ell }\{|a_{\ell m}|^{2}-1\}+R_{\ell }(I)\text{ }%
,  \label{theo}
\end{align}%
where
\begin{equation*}
\mathrm{Var}(\mathcal{N}_{\ell ;I}^{c}[2])=\ell ^{3}\nu ^{c}(I)+o(\ell ^{3})%
\text{ },\text{ }\mathbb{E}[R_{\ell }^{2}(I)]=o(\ell ^{3})\text{ },
\end{equation*}%
uniformly over $I$.
\end{theorem}

Theorem \ref{th1} entails actually two different results, namely

a) the high frequency behaviour of the number of critical points is
dominated by a single term, proportional to the second-order Wiener chaos
projection $\mathcal{N}_{\ell ;I}^{c}[2]$, and

b) the second-order Wiener chaos projection $\mathcal{N}_{\ell ;I}^{c}[2]$
admits a simple expression in terms of the integral of $H_{2}(f_{\ell }(x))$
over $\mathbb{S}^{2}.$

These results share some surprising features with the asymptotic expressions
for Lipschitz-Killing curvatures reported in (\ref{aop1}-\ref{aop3}); in
particular, while the computation of critical points by means of the
Kac-Rice formula (see below) requires the evaluation of gradient and Hessian
fields, the dominating term depends in the high-frequency regime only on the
(random) $L^{2}(\mathbb{S}^{2})$ norm of the eigenfunctions $f_{\ell }:$%
\begin{equation}
\frac{4\pi }{2\ell +1}\sum_{m=-\ell }^{\ell }\{|a_{\ell m}|^{2}-1\}=\int_{%
\mathbb{S}^{2}}H_{2}(f_{\ell }(x))dx=\int_{\mathbb{S}^{2}}f_{\ell
}^{2}(x)dx-4\pi =\left\Vert f_{\ell }\right\Vert _{L^{2}(\mathbb{S}%
^{2})}^{2}-\mathbb{E}\left\Vert f_{\ell }\right\Vert _{L^{2}(\mathbb{S}%
^{2})}^{2}\text{ .}  \label{randomnorm}
\end{equation}%
As a simple Corollary we are able to establish a quantitative Central Limit
Theorem; for this purpose, let us recall the definition of the Wasserstein
distance between probability distributions (see for instance \cite{noupebook}
and the references therein), which for any two random variables $X,Y$ is
given by%
\begin{equation*}
d_{W}(X,Y):=\sup_{h\in \text{Lip}(1)}\left\vert \mathbb{E}h(X)-\mathbb{E}%
h(Y)\right\vert ,
\end{equation*}%
where $\text{Lip}(1)$ denotes the space of Lipschitz functions with bounding
constant equal to $1$. Writing $Z\sim $ $\mathcal{N}(0,1)$ for a standard
Gaussian variable, we have that:

\begin{corollary}
\label{cor1}For $I\subset \mathbb{R}$ s.t. $\nu ^{c}(I)\neq 0,$ as $\ell
\rightarrow \infty $,
\begin{equation*}
d_{W}\left( \frac{\mathcal{N}_{\ell }^{c}(I)-\mathbb{E}[\mathcal{N}_{\ell
}^{c}(I)]}{\sqrt{\mathrm{Var}(\mathcal{N}_{\ell }^{c}(I))}},Z\right)
=o\left( 1\right) \text{ }.
\end{equation*}
\end{corollary}

The proof of the Corollary is standard, noting that by \eqref{H2} the right
hand side of \eqref{theo} is proportional to a sum of independent and
identically distributed random variables with zero mean and finite variance,
plus a remainder which is negligible in the mean square sense.

\begin{remark}
Note that for $I=[u,\infty )$ we have that%
\begin{align*}
\int_{u}^{\infty }p_{3}^{c}(t)dt & =\int_{u}^{\infty }\frac{1}{\sqrt{8\pi }}%
e^{-\frac{3}{2}t^{2}}(2-6t^{2})dt-\int_{u}^{\infty }\frac{1}{\sqrt{8\pi }}%
e^{-\frac{1}{2}t^{2}}(1-4t^{2}+t^{4})dt \\
& =-\frac{1}{2}\frac{\sqrt{2}}{\sqrt{\pi }}ue^{-\frac{3}{2}u^{2}}+\frac{1}{4}%
\frac{\sqrt{2}}{\sqrt{\pi }}ue^{-\frac{1}{2}u^{2}}\left( u^{2}-1\right),
\end{align*}%
and hence
\begin{equation*}
\mathcal{N}_{\ell }^{c}(I)-\mathbb{E}[\mathcal{N}_{\ell }^{c}(I)]=\frac{%
\lambda _{\ell }}{2}\,\left\{ -\frac{\sqrt{2}}{\sqrt{3}}H_{1}\left(\frac{%
\sqrt{3}}{\sqrt{2}}u \right)\phi \left(\frac{\sqrt{3}}{\sqrt{2}}u\right)+%
\frac{1}{2}H_{1}(u)H_{2}(u)\phi (u)\right\} \frac{1}{2\pi }\int_{\mathbb{S}%
^{2}}H_{2}(f_{\ell }(x))dx+o_{p}(\ell ^{3/2}).
\end{equation*}
\end{remark}

\subsection{Discussion}

Theorem \ref{th1} yields one more reduction principle for geometric
functionals of random spherical harmonics, entailing in particular the
asymptotically full correlation%
\begin{equation*}
\lim_{\ell \rightarrow \infty }\text{Corr}\left\{ \mathcal{N}_{\ell }^{c}(I),%
\mathcal{L}_{k}(A_{u}(f_{\ell };\mathbb{S}^{2}))\right\} =1\text{ for all }%
u\neq 0,\;I\text{ such that }\int_{I}p_{3}^{c}(t)dt\neq 0,\;\text{ and }%
k=0,1,2.
\end{equation*}%
It should be noted that $\int_{I}p_{3}^{c}(t)dt\neq 0$ holds for all
half-intervals of the form $[u,\infty )$, $u\neq 0,\pm \overline{u},$ where $%
\overline{u}\simeq 1.\,\allowbreak 209\,6$ solves $-2e^{-\overline{u}%
^{2}}+\left( \overline{u}^{2}-1\right) =0$. In other words, knowledge of the
number of critical points on any excursion set $A_{u}$ for $u\neq 0,\pm
\overline{u}$ allows to fully characterize (at least in the high-energy
limit) the geometry of these excursion sets, i.e., their area, their
boundary length, and their Euler-Poincar\'{e} characteristic. As mentioned
earlier, it is also remarkable that, while the computation of the number of
critical points requires the evaluation of higher-order derivatives, the
asymptotic expression given in Theorem 1 depends on the $L^{2}$ norm of $f_{\ell }$ and no extra information.

An interesting open question is the characterization of the behaviour of
critical points for intervals $I$ such that the integral of $p_{3}^{c}(.)$
vanishes, the most interesting case being clearly $I=\mathbb{R},$ i.e., the
total number of critical points. A heuristic rationale explaining why the
variance of the total number of critical points is asymptotically an order
of magnitude smaller (up to logarithmic factors) than for \textquotedblleft
typical" intervals $I$ can be given as follows; from (\ref{randomnorm}),
Theorem \ref{th1} can be viewed as stating that the fluctuations in the
number of critical points over a \textquotedblleft generic" interval $I$ are
proportional to the fluctuation in the random norm of the eigenfunctions.
Clearly this cannot be the case for $I=\mathbb{R};$ indeed, the total number
of critical points for a given realization $f_{\ell }$ is independent from
any scaling factor, including the $L^{2}(\mathbb{S}^{2})$ norm of the
eigenfunctions. This leaves open the question about the asymptotic
distribution for this total number $\mathcal{N}_{\ell }^{c}$; by analogy
with some recent results by \cite{MRW}, we conjecture that the following
expression holds:
\begin{equation*}
\mathcal{N}_{\ell }^{c}(I)-\mathbb{E}[\mathcal{N}_{\ell }^{c}(I)]=-\mu
^{c}(I)\ell ^{2}\frac{1}{4!}\int_{\mathbb{S}^{2}}H_{4}(f_{\ell
}(x))dx+O_{p}(\ell ^{2}).
\end{equation*}%
Note that%
\begin{eqnarray*}
\text{Var}\left( -\mu ^{c}(I)\ell ^{2}\frac{1}{4!}\int_{\mathbb{S}%
^{2}}H_{4}(f_{\ell }(x))dx\right) &=&\left[ \mu ^{c}(I)\right] ^{2}\ell ^{4}%
\frac{1}{576}\text{Var}\left( \int_{\mathbb{S}^{2}}H_{4}(f_{\ell
}(x))dx\right) \\
&=&\left[ \mu ^{c}(I)\right] ^{2}\ell ^{4}\frac{1}{576}\frac{576\log \ell }{%
\ell ^{2}}+O(\ell ^{2})=\left[ \mu ^{c}(I)\right] ^{2}\ell ^{2}\log \ell
+O(\ell ^{2}),
\end{eqnarray*}%
consistent with \eqref{fourthchaos}. In particular, for the total number of
critical points we conjecture the asymptotic equivalence%
\begin{equation*}
\mathcal{N}_{\ell }^{c}-\mathbb{E}[\mathcal{N}_{\ell }^{c}]=-\frac{1}{%
3^{3/2}\pi }\ell ^{2}\frac{1}{4!}\int_{\mathbb{S}^{2}}H_{4}(f_{\ell
}(x))dx+O_{p}(\ell ^{2}),
\end{equation*}%
the variance of the right-hand side being consistent with the result given
in \cite{CW}, where is shown that%
\begin{equation*}
\text{Var}(\mathcal{N}_{\ell }^{c})=\frac{\ell ^{2}\log \ell }{3^{3}\pi ^{2}}%
+O(\ell ^{2}).
\end{equation*}%
The investigation of this conjecture is left as a topic for further research.

\subsection{Plan of the paper}

As mentioned above, our proof requires two main ingredients, i.e., the
Kac-Rice formula to express the number of critical points as a local
functional of gradient and Hessian, and its expansion into Hermite
polynomials/Wiener-Ito chaoses. For a correct implementation of the Kac-Rice
formula, our first step is to review in Section \ref{GradHess} some
background differential geometry material on the gradient and Hessian
fields, and to compute their covariances; the properties of the resulting
covariance matrices are then established in Section \ref{cov_matx_ev}, where
it is shown in particular that the covariance function for the gradient
vector of random eigenfunctions evaluated at any two arbitrary points on the
sphere is non-singular. These results are then used in Section \ref{Kac-Rice}
to prove the validity (in the $L^{2}(\Omega )$ sense) of the expansion for
the Kac-Rice formula into Wiener chaoses, a technique exploited in other
recent papers on geometric functionals of Gaussian eigenfunctions, for
instance also in \cite{MW}, \cite{MR2015}, \cite{MPRW2015}, \cite{dalmao},
\cite{peccatirossi}, \cite{BMW}, \cite{CM2018}, \cite{MRW}. Finally, in
Section \ref{Sec:Proof} the expansion is analytically computed and the
simple dominating term is derived. A number of technical Lemmas related to
computations of covariances and conditional expected values are collected in
the Appendix.

\section{Gradient and Hessian Fields \label{GradHess}}

The proper computation of covariance matrices requires some careful
discussion on (standard) background material in differential geometry. For $%
x=(\theta _{x},\varphi _{x})\in \mathbb{S}^{2}\setminus \{N,S\}$ ($N,S$ are
the north and south poles i.e. $\theta =0$ and $\theta =\pi $ respectively),
the vectors
\begin{equation*}
\partial _{1;x}=\frac{\partial }{\partial \theta }\Big\vert_{\theta =\theta
_{x}},\hspace{1cm}\partial _{2;x}=\frac{1}{\sin \theta }\frac{\partial }{%
\partial \varphi }\Big\vert_{\theta =\theta _{x},\varphi =\varphi _{x}},
\end{equation*}%
constitute an orthonormal basis for the tangent plane $T_{x}(\mathbb{S}^{2})$%
; in these coordinates the gradient is given by $\nabla =(\frac{\partial }{%
\partial \theta },\frac{1}{\sin \theta }\frac{\partial }{\partial \varphi })$%
. For second-order derivatives, we shall use the following notation:%
\begin{equation*}
\partial _{11}:=\frac{\partial ^{2}}{\partial \theta ^{2}},\;\;\;\partial
_{21}:=\frac{\partial ^{2}}{\sin \theta \partial \varphi \partial \theta }%
,\;\;\;\partial _{22}:=\frac{\partial ^{2}}{\sin ^{2}\theta \partial \varphi
^{2}}.
\end{equation*}%
Let us now recall that the covariant Hessian of a function $f\in C^{2}(%
\mathbb{S}^{2})$ is the bilinear symmetric map from $C^{1}(T(\mathbb{S}%
^{2}))\times C^{1}(T(\mathbb{S}^{2}))$ to $C^{0}(\mathbb{S}^{2})$ defined by
\begin{equation*}
(\nabla ^{2}f)(X,Y)=XYf-\nabla _{X}Yf,\hspace{0.5cm}X,Y\in T(\mathbb{S}^{2}),
\end{equation*}%
where $\nabla _{X}$ denotes Levi-Civita connection, see \cite{adlertaylor},
Chapter 7 or the Appendix below for some details and definitions. For our
computations to follow we shall need the matrix-valued process $\nabla ^{2}{%
f_{\ell }}(x)$ with elements given by $\left\{ (\nabla ^{2}f)(\partial
_{a},\partial _{b})\right\} _{a,b=1,2}$; in coordinates as above, this
matrix can be expressed as (see Appendix A for some more details)%
\begin{eqnarray*}
\nabla ^{2}f_{\ell }(x) &=&\left[
\begin{matrix}
\frac{\partial ^{2}f_{\ell }(x)}{\partial \theta ^{2}} & \frac{1}{\sin
\theta _{x}}[\frac{\partial ^{2}f_{\ell }(x)}{\partial \theta \partial
\varphi }-\frac{\cos \theta _{x}}{\sin \theta _{x}}\frac{\partial f_{\ell
}(x)}{\partial \varphi }] \\
\frac{1}{\sin \theta _{x}}[\frac{\partial ^{2}f_{\ell }(x)}{\partial \theta
\partial \varphi }-\frac{\cos \theta _{x}}{\sin \theta _{x}}\frac{\partial
f_{\ell }(x)}{\partial \varphi }] & \frac{1}{\sin ^{2}\theta _{x}}[\frac{%
\partial ^{2}f_{\ell }(x)}{\partial \varphi ^{2}}+\sin \theta _{x}\cos
\theta \frac{\partial f_{\ell }(x)}{\partial \theta }]%
\end{matrix}%
\right] \\
&=&\left[
\begin{matrix}
\partial _{11}f_{\ell }(x) & (\partial _{21}-\cot \theta _{x}\;\partial
_{2})f_{\ell }(x) \\
(\partial _{21}-\cot \theta _{x}\;\partial _{2})f_{\ell }(x) & (\partial
_{22}+\cot \theta _{x}\;\partial _{1})f_{\ell }(x)%
\end{matrix}%
\right] .
\end{eqnarray*}%
We write as usual $vec(\nabla ^{2}f_{\ell })$ for the column vector stacking
the different elements of $(\nabla ^{2}f_{\ell }),$ i.e.,
\begin{equation*}
vec(\nabla ^{2}f_{\ell })=\left(
\begin{array}{c}
\partial _{11}f_{\ell }(x) \\
\partial _{21}f_{\ell }(x)-\cot \theta _{x}\;\partial _{2}f_{\ell }(x) \\
\partial _{22}f_{\ell }(x)+\cot \theta _{x}\;\partial _{1}f_{\ell }(x)%
\end{array}%
\right) .
\end{equation*}%
The next result gives the exact covariance matrix for the five-dimensional
vector including the elements of the gradient and the (covariant) Hessian:

\begin{proposition}
For every point $x\in \mathbb{S}^{2}\backslash \left\{ N;S\right\}$, the
distribution of the $5$-dimensional vector $(\nabla f_{\ell }, vec(\nabla
^{2}f_{\ell }))$ is zero-mean Gaussian, with variance-covariance matrix
\begin{equation}
\left(
\begin{array}{ccccc}
P_{\ell }^{\prime }(1) & 0 & 0 & 0 & 0 \\
0 & P_{\ell }^{\prime }(1) & 0 & 0 & 0 \\
0 & 0 & 3P_{\ell }^{\prime \prime }(1)+P_{\ell }^{\prime }(1) & 0 & P_{\ell
}^{\prime \prime }(1)+P_{\ell }^{\prime }(1) \\
0 & 0 & 0 & P_{\ell }^{\prime \prime }(1) & 0 \\
0 & 0 & P_{\ell }^{\prime \prime }(1)+P_{\ell }^{\prime }(1) & 0 & 3P_{\ell
}^{\prime \prime }(1)+P_{\ell }^{\prime }(1)%
\end{array}%
\right).  \label{CovMat}
\end{equation}
\end{proposition}

\begin{proof}
The following results are proved in Lemma \ref{1-1;2-2;1-2} in Appendix B:
\begin{equation*}
\text{Var}(\partial _{1}f_{\ell }(x))=\text{Var}(\partial _{2}f_{\ell
}(x))=P_{\ell }^{\prime }(1), \;\;\; \text{Cov}(\partial _{1}f_{\ell
}(x),\partial _{1}f_{\ell }(x))=0,
\end{equation*}%
for the higher order derivatives, we have (see Lemmas \ref{11-11}, \ref%
{21-21}, \ref{22-22}, \ref{22-11}, \ref{11-21}, \ref{21-22})
\begin{align*}
\text{Var}(\partial _{11}f_{\ell }(x)) &=3P_{\ell }^{\prime \prime
}(1)+P_{\ell }^{\prime }(1), \\
\text{Var}(\partial _{21}f_{\ell }(x)) &=P_{\ell }^{\prime \prime }(1)+\cot
^{2}\theta _{x}P_{\ell }^{\prime }(1), \\
\text{Var}(\partial _{22}f_{\ell }(x)) &=3P_{\ell }^{\prime \prime
}(1)+P_{\ell }^{\prime }(1)-\cot ^{2}\theta _{x} \, P_{\ell }^{\prime }(1),
\end{align*}%
and
\begin{align*}
\text{Cov}(\partial _{22}f_{\ell }(x),\partial _{11}f_{\ell }(x)) &=P_{\ell
}^{\prime \prime }(1)+P_{\ell }^{\prime }(1), \\
\text{Cov}(\partial _{22}f_{\ell }(x),\partial _{21}f_{\ell }(x)) &=\text{Cov%
}(\partial _{11}f_{\ell }(x),\partial _{21}f_{\ell }(x))=0.
\end{align*}%
Moreover in Lemmas \ref{11-1;11-2}, \ref{12-1;12-2}, \ref{22-1;22-2}, it is
shown that%
\begin{align*}
\text{Cov}(\partial _{11}f_{\ell }(x),\partial _{1}f_{\ell }(x)) &=\text{Cov}%
(\partial _{11}f_{\ell }(x),\partial _{2}f_{\ell }(x))=0, \\
\text{Cov}(\partial _{21}f_{\ell }(x),\partial _{1}f_{\ell }(x)) &=0,\;\;
\text{Cov}(\partial _{21}f_{\ell }(x),\partial _{2}f_{\ell }(x))=\cot
\theta_{x} \, P_{\ell }^{\prime }(1), \\
\text{Cov}(\partial _{22}f_{\ell }(x),\partial _{1}f_{\ell }(x)) &=-\cot
\theta_{x}\, P_{\ell }^{\prime }(1), \;\; \text{Cov}(\partial _{22}f_{\ell
}(x),\partial_{2}f_{\ell }(x))=0.
\end{align*}%
The remaining computations are all straightforward - it is interesting,
however, to note the delicate cancellations that occur. For instance, we
have, for any $x\in \mathbb{S}^{2}$,
\begin{equation*}
\text{Var}\left(\frac{1}{\sin ^{2}\theta _{x}}\frac{\partial ^{2}f_{\ell }(x)%
}{\partial \varphi ^{2}}+\cot \theta _{x}\frac{\partial f_{\ell }(x)}{%
\partial \theta }\right)
\end{equation*}%
\begin{align*}
&=\text{Var}(\partial _{22}f_{\ell }(x))+\cot ^{2}\theta _{x}Var(\partial
_{1}f_{\ell }(x))+2\cot \theta _{x} \text{Cov}(\partial _{22}f_{\ell
}(x),\partial _{2}f_{\ell }(x)) \\
&=3P_{\ell }^{\prime \prime }(1)+P_{\ell }^{\prime }(1)+\cot ^{2}\theta
_{x}P_{\ell }^{\prime }(1)+\cot ^{2}\theta _{x}P_{\ell }^{\prime }(1)-2\cot
^{2}\theta _{x}P_{\ell }^{\prime }(1) \\
&=3P_{\ell }^{\prime \prime }(1)+P_{\ell }^{\prime }(1).
\end{align*}%
Likewise%
\begin{equation*}
\text{Var}\left(\frac{1}{\sin \theta }\frac{\partial ^{2}f_{\ell }(x)}{%
\partial \theta \partial \varphi }-\cot \theta _{x}\frac{1}{\sin \theta _{x}}%
\frac{\partial f_{\ell }(x)}{\partial \varphi }\right)
\end{equation*}%
\begin{align*}
&=\text{Var}(\partial _{12}f_{\ell }(x)+\cot ^{2}\theta _{x} \text{Var}%
(\partial _{2}f_{\ell }(x))-2\cot \theta _{x} \text{Cov}(\partial
_{21}f_{\ell }(x),\partial _{2}f_{\ell }(x)) \\
&=P_{\ell }^{\prime \prime }(1)+\cot ^{2}\theta _{x}P_{\ell }^{\prime
}(1)+\cot ^{2}\theta _{x}P_{\ell }^{\prime }(1)-2\cot ^{2}\theta _{x}P_{\ell
}^{\prime }(1) \\
&=P_{\ell }^{\prime \prime }(1).
\end{align*}
All the remaining terms are immediate consequences of the Lemmas.
\end{proof}

\begin{remark}
It follows from the previous proposition that the elements of the covariant
Hessian (the second-order covariant derivatives) have finite and constant
variance for all locations on the sphere. It should be stressed that this is
not the case for the standard derivatives, i.e. the elements of the iterated
gradient $\nabla ^{2}f_{\ell };$ in fact, we have shown above that%
\begin{align*}
\text{Var}(\partial _{11}f_{\ell }(x)) &=3P_{\ell }^{\prime \prime
}(1)+P_{\ell }^{\prime }(1), \\
\text{Var}(\partial _{21}f_{\ell }(x)) &=P_{\ell }^{\prime \prime }(1)+2\cot
^{2}\theta _{x}P_{\ell }^{\prime }(1), \\
\text{Var}(\partial _{22}f_{\ell }(x)) &=3P_{\ell }^{\prime \prime
}(1)+P_{\ell }^{\prime }(1)-\cot ^{2}\theta _{x}P_{\ell }^{\prime }(1),
\end{align*}%
so that the variance of $(\partial _{12}f_{\ell }(x),\partial _{22}f_{\ell
}(x))$ is not constant over $\mathbb{S}^{2}$ - and indeed not even bounded,
so that $L^{2}$ expansions would not be feasible. On the other hand, in the
element of the (covariant) Hessian the extra terms introduced by means of
the Levi-Civita connection and the Christoffel symbols ensure the exact
cancellation of the location-dependent factors. Note also that while
covariant derivatives of different orders are zero when evaluated on the
same point, this is not the case for standard derivatives, indeed both
\begin{equation*}
\text{Cov}(\partial _{21}f_{\ell }(x),\partial _{2}f_{\ell }(x))=-\text{Cov}%
(\partial _{22}f_{\ell }(x),\partial _{1}f_{\ell }(x))=\cot \theta
_{x}P_{\ell }^{\prime }(1).
\end{equation*}%
These covariances are zero only for $\theta _{x}=\frac{\pi }{2},$ where the
tangent plane has a ``Euclidean" basis $\frac{\partial }{\partial \theta },%
\frac{\partial }{\partial \varphi }$ and the covariant derivatives take the
same values as the standard ones.
\end{remark}

\begin{remark}
Note that each one of the three quantities%
\begin{equation*}
\text{Var}\left(\frac{\partial ^{2}f_{\ell }}{\partial \theta ^{2}}%
\right),\;\;\; \text{Var}\left(\frac{1}{\sin ^{2}\theta }\frac{\partial
^{2}f_{\ell }}{\partial \varphi ^{2}}+\cot \theta \frac{\partial f_{\ell }}{%
\partial \theta }\right),\;\;\; \text{Cov}\left(\frac{\partial ^{2}f_{\ell }%
}{\partial \theta ^{2}},\frac{1}{\sin ^{2}\theta }\frac{\partial ^{2}f_{\ell
}}{\partial \varphi ^{2}}+\cot \theta \frac{\partial f_{\ell }}{\partial
\theta }\right),
\end{equation*}%
could actually be evaluated immediately from the other two. Indeed,
recalling that the trace of the Hessian matrix is indeed the spherical
Laplacian%
\begin{equation*}
\Delta _{\mathbb{S}^{2}}f_{\ell }(x)=\frac{\partial ^{2}f_{\ell }}{\partial
\theta ^{2}}+\frac{1}{\sin ^{2}\theta }\frac{\partial ^{2}f_{\ell }}{%
\partial \varphi ^{2}}+\cot \theta \frac{\partial f_{\ell }}{\partial \theta
}=-\lambda _{\ell }f_{\ell }(x),
\end{equation*}%
we have the identity%
\begin{equation*}
\lambda _{\ell }^{2}=\text{Var}(-\lambda _{\ell }f_{\ell })=\text{Var}\left(%
\frac{\partial ^{2}f_{\ell }}{\partial \theta ^{2}}+\frac{1}{\sin ^{2}\theta
}\frac{\partial ^{2}f_{\ell }}{\partial \varphi ^{2}}+\cot \theta \frac{%
\partial f_{\ell }}{\partial \theta }\right)
\end{equation*}%
\begin{equation*}
=\text{Var}\left(\frac{\partial ^{2}f_{\ell }}{\partial \theta ^{2}}\right)+%
\text{Var}\left(\frac{1}{\sin ^{2}\theta }\frac{\partial ^{2}f_{\ell }}{%
\partial \varphi ^{2}}+\cot \theta \frac{\partial f_{\ell }}{\partial \theta
}\right)+2 \text{Cov}\left(\frac{\partial ^{2}f_{\ell }}{\partial \theta ^{2}%
},\frac{1}{\sin ^{2}\theta }\frac{\partial ^{2}f_{\ell }}{\partial \varphi
^{2}}+\cot \theta \frac{\partial f_{\ell }}{\partial \theta }\right),
\end{equation*}%
whence, for instance, using%
\begin{equation*}
\text{Var}\left(\frac{\partial ^{2}f_{\ell }}{\partial \theta ^{2}}%
\right)=3P_{\ell }^{\prime \prime }(1)+P_{\ell }^{\prime }(1),
\end{equation*}%
\begin{equation*}
\text{Cov}\left(\frac{\partial ^{2}f_{\ell }}{\partial \theta ^{2}},\frac{1}{%
\sin ^{2}\theta }\frac{\partial ^{2}f_{\ell }}{\partial \varphi ^{2}}+\sin
\theta \cos \theta \frac{\partial f_{\ell }}{\partial \theta }%
\right)=P_{\ell }^{\prime \prime }(1)+P_{\ell }^{\prime }(1),
\end{equation*}%
we obtain immediately%
\begin{equation*}
\text{Var}\left(\frac{1}{\sin ^{2}\theta }\frac{\partial ^{2}f_{\ell }}{%
\partial \varphi ^{2}}+\cot \theta \frac{\partial f_{\ell }}{\partial \theta
}\right)
\end{equation*}%
\begin{align*}
&=\lambda _{\ell }^{2}-3P_{\ell }^{\prime \prime }(1)-P_{\ell }^{\prime
}(1)-2P_{\ell }^{\prime \prime }(1)-2P_{\ell }^{\prime }(1)=\lambda _{\ell
}^{2}-5P_{\ell }^{\prime \prime }(1)-3P_{\ell }^{\prime }(1) \\
&=\lambda _{\ell }^{2}-5\frac{\lambda _{\ell }}{8}(\lambda _{\ell }-2)-\frac{%
3}{2}\lambda _{\ell }=3\frac{\lambda _{\ell }}{8}(\lambda _{\ell }-2)+\frac{%
\lambda _{\ell }}{2} \\
&=3P_{\ell }^{\prime \prime }(1)+P_{\ell }^{\prime }(1)\text{ .}
\end{align*}
\end{remark}

\section{Covariance Matrices\label{cov_matx_ev}}

Recall first that, since the $f_{\ell }$ are eigenfunctions of the spherical
Laplacian, hence, for every $x\in \mathbb{S}^{2}$, we can write
\begin{equation}
f_{\ell }(x)=-\frac{\Delta _{\mathbb{S}^{2}}f_{\ell }(x)}{\lambda _{\ell }};
\label{linear_dp}
\end{equation}%
note that at the critical points we have $\Delta _{\mathbb{S}^{2}}f_{\ell
}=\partial _{11}f_{\ell }+\partial _{22}f_{\ell }$, whence their number with
value in $I$ is, by symmetry, given by
\begin{equation*}
\mathcal{N}_{\ell }^{c}(I)=\#\{x\in \mathbb{S}^{2}:\; \frac{\Delta _{\mathbb{%
S}^{2}}f_{\ell }(x)}{\lambda _{\ell }}\in I,\;\;\nabla f_{\ell
}(x)=0\}=\#\{x\in \mathbb{S}^{2}:\; \frac{\partial _{11}f_{\ell }+\partial
_{22}f_{\ell }}{\lambda _{\ell }}\in I,\;\;\nabla f_{\ell }(x)=0\}.
\end{equation*}

\noindent For every $x\in \mathbb{S}^{2}$, let us now denote by $\Sigma
_{\ell }(x,y)$ the covariance matrix for the $10$-dimensional Gaussian
random vector
\begin{equation*}
(\nabla f_{\ell }(x),\nabla f_{\ell }(y),vec(\nabla ^{2}f_{\ell
}(x)),vec(\nabla ^{2}f_{\ell }(y))),
\end{equation*}%
which combines the gradient and the elements of the Hessian evaluated at $x$%
, $y$; we shall write
\begin{equation*}
\Sigma _{\ell }(x,y)=\left(
\begin{array}{cc}
A_{\ell }(x,y) & B_{\ell }(x,y) \\
B_{\ell }^{T}(x,y) & C_{\ell }(x,y)%
\end{array}%
\right),
\end{equation*}%
where the $A_{\ell }$ and $C_{\ell }$ components collect the variances of
the gradient and Hessian terms respectively, while the matrix $B_{\ell }$
collects the covariances between first and second order covariant
derivatives. More explicitly, we have that
\begin{equation*}
A_{\ell }(x,y)_{4\times 4}=\left(
\begin{array}{cc}
a_{\ell }(x,x) & a_{\ell }(x,y) \\
a_{\ell }(y,x) & a_{\ell }(y,y)%
\end{array}%
\right) =\left. \mathbb{E}\left[ \left(
\begin{array}{c}
\nabla f_{\ell }(\bar{x})^{T} \\
\nabla f_{\ell }(\bar{y})^{T}%
\end{array}%
\right) \left(
\begin{array}{cc}
\nabla f_{\ell }(x) & \nabla f_{\ell }(y)%
\end{array}%
\right) \right] \right\vert _{{x=\bar{x}},{y=\bar{y}}},
\end{equation*}%
and more precisely
\begin{equation*}
a_{\ell }(x,x)=a_{\ell }(y,y)=\left(
\begin{array}{cc}
P_{\ell }^{\prime }(1) & 0 \\
0 & P_{\ell }^{\prime }(1)%
\end{array}%
\right).
\end{equation*}%
Recalling also that $P_{\ell }^{\prime }(1)=\frac{\ell (\ell +1)}{2}$, for $%
\lambda _{\ell }=\ell (\ell +1)$, we have%
\begin{equation*}
A_{\ell }(x,y)=\left(
\begin{array}{cccc}
\frac{\lambda _{\ell }}{2} & 0 & \alpha _{1,\ell }(x,y) & 0 \\
0 & \frac{\lambda _{\ell }}{2} & 0 & \alpha _{2,\ell }(x,y) \\
\alpha _{1,\ell }(x,y) & 0 & \frac{\lambda _{\ell }}{2} & 0 \\
0 & \alpha _{2,\ell }(x,y) & 0 & \frac{\lambda _{\ell }}{2}%
\end{array}%
\right) .
\end{equation*}

\noindent The matrix $B_{\ell }$ collects the covariances between first and
second order derivatives, and is given by
\begin{align*}
B_{\ell }(x,y)_{4\times 6}& =\left. \mathbb{E}\left[ \left(
\begin{array}{c}
vec(\nabla f_{\ell }(\bar{x})) \\
vec(\nabla f_{\ell }(\bar{y}))%
\end{array}%
\right) \left(
\begin{array}{cc}
vec(\nabla ^{2}f_{\ell }(x))^{T} & vec(\nabla ^{2}f_{\ell }(y))^{T}%
\end{array}%
\right) \right] \right\vert _{{x=\bar{x}},{y=\bar{y}}} \\
& =\left(
\begin{array}{cc}
b_{\ell }(x,x) & b_{\ell }(x,y) \\
b_{\ell }(y,x) & b_{\ell }(y,y)%
\end{array}%
\right) .
\end{align*}%
It is well-known that for Gaussian isotropic processes, for $i,j=1,2$, the
second derivatives $e_{i}^{x}e_{j}^{x}f_{\ell }(x)$ are independent of $%
e_{i}^{x}f_{\ell }(x)$ at every fixed point $x\in \mathbb{S}^{2}$ see, e.g.,
\cite{adlertaylor} Section 5.5; we have then
\begin{equation*}
b_{\ell }(x,x)=b_{\ell }(y,y)=\left(
\begin{array}{ccc}
0 & 0 & 0 \\
0 & 0 & 0%
\end{array}%
\right).
\end{equation*}

\noindent Finally, the matrix $C_{\ell }$ contains the variances of
second-order derivatives, and we have
\begin{equation*}
C_{\ell }(x,y)_{6\times 6}=\left. \mathbb{E}\left[
\begin{array}{c}
\left(
\begin{array}{c}
\nabla ^{2}f_{\ell }(\bar{x})^{T} \\
\nabla ^{2}f_{\ell }(\bar{y})^{T}%
\end{array}%
\right) \left(
\begin{array}{cc}
\nabla ^{2}f_{\ell }({x}) & \nabla ^{2}f_{\ell }(\bar{y})%
\end{array}%
\right)%
\end{array}%
\right] \right\vert _{{x=\bar{x}},{y=\bar{y}}}=\left(
\begin{array}{cc}
c_{\ell }(x,x) & c_{\ell }(x,y) \\
c_{\ell }(y,x) & c_{\ell }(y,y)%
\end{array}%
\right) .
\end{equation*}%
From direct calculations and the formula $P_{\ell }^{\prime \prime }(1)=%
\frac{\lambda _{\ell }}{8}(\lambda _{\ell }-2)$, it immediately follows that
\begin{align*}
c_{\ell }(x,x)& =\left(
\begin{array}{ccc}
3P_{\ell }^{\prime \prime }(1)+P_{\ell }^{\prime }(1) & 0 & P_{\ell
}^{\prime \prime }(1)+P_{\ell }^{\prime }(1) \\
0 & P_{\ell }^{\prime \prime }(1) & 0 \\
P_{\ell }^{\prime \prime }(1)+P_{\ell }^{\prime }(1) & 0 & 3P_{\ell
}^{\prime \prime }(1)+P_{\ell }^{\prime }(1)%
\end{array}%
\right) \\
& =\left(
\begin{array}{ccc}
\frac{\lambda _{\ell }}{8}[3\lambda _{\ell }-2] & 0 & \frac{\lambda _{\ell }%
}{8}[\lambda _{\ell }+2] \\
0 & \frac{\lambda _{\ell }}{8}[\lambda _{\ell }-2] & 0 \\
\frac{\lambda _{\ell }}{8}[\lambda _{\ell }+2] & 0 & \frac{\lambda _{\ell }}{%
8}[3\lambda _{\ell }-2]%
\end{array}%
\right) =c_{\ell }(y,y).
\end{align*}

The following Proposition is really a special case of Lemma C1 in \cite%
{WigmanJMP2009}; nevertheless we include a short proof for completeness:

\begin{proposition}
\label{Thecondition} (See \cite{WigmanJMP2009}) For every $(x,y)\in \mathbb{S%
}^{2}$ such that $d_{\mathbb{S}^{2}}(x,y)\neq 0,\pi $, the Gaussian vector $%
(\nabla f_{\ell }(x),\nabla f_{\ell }(y))$ has a non-degenerate density
function, i.e., the covariance matrix $A_{\ell }(x,y)$ is invertible.
\end{proposition}

\begin{proof}
We compute the covariance matrix at the equator, where our choice of
coordinates is such that both the vectors $\partial _{1}$ and $\partial _{2}$
are moved by parallel transport from $x$ to $y.$ Given however any arbitrary
orientations in $x$ and $y$, we would obtain the matrix $R(\psi )A_{\ell
}(x,y)R(\psi )^{\prime },$ where%
\begin{equation*}
R(\psi )=\left(
\begin{array}{cccc}
1 & 0 & 0 & 0 \\
0 & 1 & 0 & 0 \\
0 & 0 & \cos \psi  & -\sin \psi  \\
0 & 0 & \sin \psi  & \cos \psi
\end{array}%
\right)
\end{equation*}%
is the matrix that describes the rotation between $\partial _{a,x}$ and $%
\partial _{b,y}$ after transportation, $a,b=1,2$ (i.e., $\psi $ is the angle
that occurs between $\partial _{1,x}$ when it is transported to $y$ along a
geodesic in such a way that it remains parallel and the vector $\partial
_{1,y};$ this angle is the same for both vectors because the basis are
orthonormal). Clearly the matrices $R$ are full rank and hence have no
impact on the rank.

We are now in the position to describe explicitly the symmetric matrix $%
A_{\ell }(x,y);$ for notational simplicity and without loss of generality,
we focus on pairs of points lying on the \textquotedblleft equator" $\theta
_{x}=\theta _{y}=\frac{\pi }{2},$ where we obtain
\begin{equation*}
A_{\ell }(x,y)=\left(
\begin{array}{cccc}
P_{\ell }^{\prime }(1) & \ast & \ast & \ast \\
0 & P_{\ell }^{\prime }(1) & \ast & \ast \\
P_{\ell }^{\prime }(\left\langle x,y\right\rangle ) & 0 & P_{\ell }^{\prime
}(1) & \ast \\
0 & -P_{\ell }^{\prime \prime }(\left\langle x,y\right\rangle )\sin
^{2}(\varphi _{x}-\varphi _{y})+P_{\ell }^{\prime }(\left\langle
x,y\right\rangle )\cos (\varphi _{x}-\varphi _{y}) & 0 & P_{\ell }^{\prime
}(1)%
\end{array}%
\right)
\end{equation*}%
with determinant equal to%
\begin{align*}
& \left\{ P_{\ell }^{\prime }(1)\right\} ^{2}\left[ \left\{ P_{\ell
}^{\prime }(1)\right\} ^{2}-\left\{ P_{\ell }^{\prime \prime }(\left\langle
x,y\right\rangle )\sin ^{2}(\varphi _{x}-\varphi _{y})+P_{\ell }^{\prime
}(\left\langle x,y\right\rangle )\cos (\varphi _{x}-\varphi _{y})\right\}
^{2}\right] \\
& -\left\{ P_{\ell }^{\prime }(\left\langle x,y\right\rangle )\right\} ^{2}
\left[ \left\{ P_{\ell }^{\prime }(1)\right\} ^{2}-\left\{ P_{\ell }^{\prime
\prime }(\left\langle x,y\right\rangle )\sin ^{2}(\varphi _{x}-\varphi
_{y})+P_{\ell }^{\prime }(\left\langle x,y\right\rangle )\cos (\varphi
_{x}-\varphi _{y})\right\} ^{2}\right] \\
& =\left[ \left\{ P_{\ell }^{\prime }(1)\right\} ^{2}-\left[ P_{\ell
}^{\prime }(\left\langle x,y\right\rangle )\right\} ^{2}\right] \left\{
\left\{ P_{\ell }^{\prime }(1)\right\} ^{2}-\left\{ P_{\ell }^{\prime \prime
}(\left\langle x,y\right\rangle )\sin ^{2}(\varphi _{x}-\varphi
_{y})+P_{\ell }^{\prime }(\left\langle x,y\right\rangle )\cos (\varphi
_{x}-\varphi _{y})\right\} ^{2}\right] .
\end{align*}%
Now%
\begin{equation*}
\left\{ P_{\ell }^{\prime }(1)\right\} ^{2}-\left\{ P_{\ell }^{\prime
}(\left\langle x,y\right\rangle )\right\} ^{2}>0\text{ for all }x\neq y,
\end{equation*}%
by standard properties of Legendre polynomials derivatives, which have a
unique maximum at $u=1$. To prove the latter statement, assume by
contradiction that $P_{\ell }^{\prime }(1)=P_{\ell }^{\prime }(\left\langle
x,y\right\rangle )$ for some $x\neq y;$ then we should have some $\alpha $
of modulus unity and some $\varphi $ such that
\begin{equation*}
\partial _{1,x}f_{\ell }=\sum_{m}a_{\ell m}P_{\ell m}^{\prime }(\cos \theta
)=\alpha \sum_{m}a_{\ell m}P_{\ell m}^{\prime }(\cos \theta )\exp (im\varphi
)=\alpha \;\partial _{1,y}f_{\ell },
\end{equation*}%
the equality holding in $L^{2}(\omega ),$ and hence with probability one;
this conclusion is clearly impossible (for $\varphi \neq 0,\pi )$, as it
requires $\exp (im\varphi )=1$ for all $m=-\ell ,...,\ell $.\emph{\ }%
Likewise,%
\begin{align*}
& \left\{ P_{\ell }^{\prime }(1)\right\} ^{2}-\left\{ P_{\ell }^{\prime
\prime }(\left\langle x,y\right\rangle )\sin ^{2}(\varphi _{x}-\varphi
_{y})+P_{\ell }^{\prime }(\left\langle x,y\right\rangle )\cos (\varphi
_{x}-\varphi _{y})\right\} ^{2} \\
& =\text{Var}(\partial _{1,x}f_{\ell })\text{Var}(\partial _{2,y}f_{\ell
})-\left\{ \text{Cov}(\partial _{2,x}f_{\ell },\partial _{2,y}f_{\ell
})\right\} ^{2}>0
\end{align*}%
unless
\begin{equation}
\partial _{2,y}f_{\ell }=\alpha \;\partial _{2,x}f_{\ell },\text{ for some }%
\alpha \text{ such that }|\alpha |=1,  \label{ident}
\end{equation}%
for some $x,y.$ However we know that, for $\theta _{x}=\theta _{y}=\frac{\pi
}{2}$,
\begin{align*}
\partial _{2,y}f_{\ell }& =\sum_{m=-\ell }^{\ell }i\,m\,a_{\ell m}\,Y_{\ell
m}(\theta _{y},\varphi _{y}) \\
& =\sum_{m=-\ell }^{\ell }i\,m\,a_{\ell m}\,e^{im\varphi _{y}}\,\sqrt{\frac{%
2\ell +1}{4\pi }}\sqrt{\frac{(\ell -m)!}{(\ell +m)!}}P_{\ell m}(0),
\end{align*}%
and
\begin{equation*}
\partial _{2,x}f_{\ell }=\sum_{m=-\ell }^{\ell }i\,m\,a_{\ell
m}\,e^{im\varphi _{x}}\sqrt{\frac{2\ell +1}{4\pi }}\sqrt{\frac{(\ell -m)!}{%
(\ell +m)!}}P_{\ell m}(0),
\end{equation*}%
whence the identity \eqref{ident} requires, for all $m$ such that $\ell +m$
is even (which ensures that $P_{\ell m}(0)\neq 0$),
\begin{equation*}
e^{im\varphi _{y}}=\alpha e^{im\varphi _{x}},
\end{equation*}%
as before impossible, unless $\varphi _{x}=\varphi _{y}.$ Hence the
statement of the Proposition follows.
\end{proof}

\section{Kac-Rice Formula and $L^{2}$-Convergence \label{Kac-Rice}}

We can now build an approximating sequence of functions (${\mathcal{N}}%
_{\ell ,\varepsilon }^{c}(I),$ say), and establish their convergence both in
the $\omega $-almost sure and in the $L^{2}(\Omega )$ sense to ${\mathcal{N}}%
_{\ell }^{c}(I)$. More precisely, let $\delta _{\varepsilon }:\mathbb{R}%
^{2}\rightarrow \mathbb{R}$ be such that
\begin{equation*}
\delta _{\varepsilon }(z):=\frac{1}{\varepsilon ^{2}}\mathbb{I}_{\{z\in
\lbrack -\varepsilon /2,\varepsilon /2]\}},
\end{equation*}%
and define the approximating sequence
\begin{equation*}
{\mathcal{N}}_{\ell ,\varepsilon }^{c}(I):=\int_{\mathbb{S}^{2}}|\text{det(}%
\nabla ^{2}f_{\ell }(x))|\mathbb{I}_{\left\{ f_{\ell }(x)\in I\right\}
}\delta _{\varepsilon }(\nabla f_{\ell }(x))dx;
\end{equation*}%
it is possible to prove the almost sure and $L^{2}(\Omega )$ convergence of $%
{\mathcal{N}}_{\ell ,\varepsilon }^{c}(I)$ to ${\mathcal{N}}_{\ell }^{c}(I)$
as $\varepsilon \rightarrow 0$:

\begin{lemma}
\label{XXe} For every $\ell \in \mathbb{N}$, we have
\begin{equation}
{\mathcal{N}}_{\ell }^{c}(I)=\lim_{\varepsilon \rightarrow 0}{\mathcal{N}}%
_{\ell ,\varepsilon }^{c}(I),  \label{Xe1}
\end{equation}%
where the convergence holds both $\omega $-a.s. and in $L^{2}(\Omega )$.
\end{lemma}

\begin{proof}
\emph{Step 1: Almost sure Convergence}. To establish the convergence $\omega
$-a.s., it is sufficient to refer to Theorem 11.2.3 in \cite{adlertaylor}.
Note that this Theorem refers to Euclidean domains, but the extension to the
spherical case can be obtained by simply referring to local maps that form
an atlas on the sphere, along the lines given for general manifolds again in
\cite{adlertaylor}, Theorem 12.1.1. Note also that these results are given
for deterministic vector functions which are assumed to be continuous and
nondegenerate (Morse), i.e. such that the Hessian does not vanish when the
gradient does; these conditions are satisfied $\omega $-a.s. for Gaussian
spherical harmonics.

\emph{Step 2: The Exact Kac-Rice Formula holds}. The fact that $\mathbb{E}%
\left[ {\mathcal{N}}_{\ell }^{c}(I)^{2}\right] <\infty $ has been shown in
\cite{CMW}, exploiting an approximate Kac-Rice formula. It is possible here
to give a stronger results; in particular, using Theorem 6.3 in \cite%
{azaiswschebor} and Proposition \ref{Thecondition}, where we have shown that
the determinant of the $4\times 4$ covariance matrix $A_{\ell }(x,y)$ of
first-order derivatives is strictly positive for every pair $(x,y)\in
\mathbb{S}^{2}\times \mathbb{S}^{2}$, we have also that the exact Kac-Rice
formula holds and hence
\begin{equation*}
\mathbb{E}\left[ {\mathcal{N}}_{\ell }^{c}(I)({\mathcal{N}}_{\ell }^{c}(I)-1)%
\right]
\end{equation*}%
\begin{equation*}
=\frac{1}{(2\pi )^{2}}\int_{\mathbb{S}^{2}\times \mathbb{S}^{2}}\mathbb{E}%
\left[ \left. |\det \left( \nabla ^{2}f_{\ell }(x)\right) |\;|\det \left(
\nabla ^{2}f_{\ell }(y)\right) |\right\vert \nabla f_{\ell }(x)=\nabla
f_{\ell }(y)=0\right] \mathbb{I}_{\{f_{\ell }(x)\in I\}}\mathbb{I}%
_{\{f_{\ell }(y)\in I\}}\frac{dxdy}{\sqrt{\det (A_{\ell }(x,y))}}.
\end{equation*}%
\emph{Step 3: The }$L^{2}$\emph{Convergence of the Approximating Sequence}.
Because we have the almost sure convergence ${\mathcal{N}}_{\ell
}^{c}(I)=\lim_{\varepsilon \rightarrow 0}{\mathcal{N}}_{\ell ,\varepsilon
}^{c}(I)$, to prove that $\mathbb{E}\left[ {\mathcal{N}}_{\ell ,\varepsilon
}^{c}(I)-{\mathcal{N}}_{\ell }^{c}(I)\right] ^{2}\rightarrow 0$ as $%
\varepsilon \rightarrow 0$ it is enough to show that
\begin{equation}
\mathbb{E}\left[ \left\{ {\mathcal{N}}_{\ell ,\varepsilon }^{c}(I)\right\}
^{2}\right] \rightarrow \mathbb{E}\left[ \left\{ {\mathcal{N}}_{\ell
}^{c}(I)\right\} ^{2}\right] \text{ as }\varepsilon \rightarrow 0.
\label{result}
\end{equation}%
Indeed we have
\begin{equation*}
\mathbb{E}\left[ \{{\mathcal{N}}_{\ell ,\varepsilon }^{c}(I)-{\mathcal{N}}%
_{\ell }^{c}(I)\}^{2}\right] =\mathbb{E}\left[ \{{\mathcal{N}}_{\ell
,\varepsilon }^{c}(I)\}^{2}\right] +\mathbb{E}\left[ \{{\mathcal{N}}_{\ell
}^{c}(I)\}^{2}\right] -2\mathbb{E}\left[ {\mathcal{N}}_{\ell ,\varepsilon
}^{c}(I){\mathcal{N}}_{\ell }^{c}(I)\right] ,
\end{equation*}%
and $\mathbb{E}\left[ {\mathcal{N}}_{\ell ,\varepsilon }^{c}(I){\mathcal{N}}%
_{\ell }^{c}(I)\right] \rightarrow \mathbb{E}\left[ \{{\mathcal{N}}_{\ell
}^{c}(I)\}^{2}\right] $ by Cauchy-Schwartz and Dominated Convergence. Recall
first that by Federer's coarea formula%
\begin{align*}
{\mathcal{N}}_{\ell ,\varepsilon }^{c}(I)& =\int_{\mathbb{S}^{2}}|\det
\left( \nabla ^{2}f_{\ell }(x)\right) |\,\delta _{\varepsilon }(\nabla
f_{\ell }(x))\,\mathbb{I}_{\{f_{\ell }(x)\in I\}}dx \\
& =\int_{\mathbb{R}^{2}}\mathcal{N}_{\ell }(u_{1},u_{2};I)\delta
_{\varepsilon }(u_{1},u_{2})du_{1}du_{2},
\end{align*}%
where%
\begin{equation*}
\mathcal{N}_{\ell }(u_{1},u_{2};I)=\text{card}\left\{ x\in \mathbb{S}%
^{2}:\nabla f_{\ell }(x)=(u_{1},u_{2}),\;f_{\ell }(x)\in I\right\} .
\end{equation*}%
Now%
\begin{equation*}
\mathbb{E}\left[ \left\{ {\mathcal{N}}_{\ell }^{c}(I)\right\} ^{2}\right]
\leq \lim_{\varepsilon \rightarrow 0}\mathbb{E}\left[ \left\{ {\mathcal{N}}%
_{\ell ,\varepsilon }^{c}(I)\right\} ^{2}\right] \text{ (by Fatou's Lemma) }
\end{equation*}%
\begin{align*}
& =\liminf_{\varepsilon \rightarrow 0}\mathbb{E}\left[ \left\{ \int_{\mathbb{%
S}^{2}}\left\vert \det \left( \nabla ^{2}f_{\ell }(x)\right) \right\vert
\delta _{\varepsilon }(\nabla f_{\ell }(x))\mathbb{I}_{\{f_{\ell }(x)\in
I\}}dx\right\} ^{2}\right] \text{ (by definition of }{\mathcal{N}}_{\ell
,\varepsilon }^{c}(I)) \\
& =\liminf_{\varepsilon \rightarrow 0}\mathbb{E}\left[ \left\{ \int_{\mathbb{%
R}^{2}}\mathcal{N}_{\ell }(u_{1},u_{2};I)\delta _{\varepsilon
}(u_{1},u_{2})du_{1}du_{2}\right\} ^{2}\right] \text{ (by Federer's coarea
formula}) \\
& \leq \limsup_{\varepsilon \rightarrow 0}\mathbb{E}\left[ \left\{ \int_{%
\mathbb{R}^{2}}\mathcal{N}_{\ell }(u_{1},u_{2};I)\delta _{\varepsilon
}(u_{1},u_{2})du_{1}du_{2}\right\} ^{2}\right] \text{ (obvious)} \\
& \leq \limsup_{\varepsilon \rightarrow 0}\int_{\mathbb{R}^{2}}\mathbb{E}%
\left[ \left\{ \mathcal{N}_{\ell }(u_{1},u_{2};I)\right\} ^{2}\right] \delta
_{\varepsilon }(u_{1},u_{2})du_{1}du_{2}\text{ (by Jensen's inequality)}.
\end{align*}%
Clearly, if we will show that the application $\mathcal{A}:\mathbb{R}%
^{2}\rightarrow \mathbb{R}$ defined by $\mathcal{A}(u_{1},u_{2}):=\mathbb{E}%
\left[ \left\{ \mathcal{N}_{\ell }(u_{1},u_{2};I)\right\} ^{2}\right] $ is
continuous in a neighbourhood of the origin, we will be able to conclude that%
\begin{equation*}
\limsup_{\varepsilon \rightarrow 0}\int_{\mathbb{R}^{2}}\mathbb{E}\left[
\left\{ \mathcal{N}_{\ell }(u_{1},u_{2};I)\right\} ^{2}\right] \delta
_{\varepsilon }(u_{1},u_{2})du_{1}du_{2}=\mathbb{E}\left[ \left\{ {\mathcal{N%
}}_{\ell }^{c}(I)\right\} ^{2}\right] ,
\end{equation*}%
and hence the result. Note however that%
\begin{equation*}
\mathbb{E}\left[ \left\{ \mathcal{N}_{\ell }(u_{1},u_{2};I)\right\} ^{2}%
\right] =\int_{\mathbb{S}^{2}\times \mathbb{S}^{2}}K_{\ell
}(x,y;u_{1},u_{2})dxdy
\end{equation*}%
where we have introduced the (generalized) two-point correlation function%
\begin{equation*}
K_{\ell }(x,y;u_{1},u_{2}):=\mathbb{E}\left[ \left. \left\vert \det \left(
\nabla ^{2}f_{\ell }(x)\right) \right\vert \left\vert \det \left( \nabla
^{2}f_{\ell }(y)\right) \right\vert \right\vert \nabla f_{\ell }(x)=\nabla
f_{\ell }(y)=(u_{1},u_{2})\right] \;p_{(\nabla f_{\ell }(x),\nabla f_{\ell
}(y))}(u_{1},u_{2}).
\end{equation*}%
Because $K_{\ell }(x,y;u_{1},u_{2})$ is a linear combination of expectations
of Gaussian moments, it is clearly continuous around $(u_{1},u_{2})=(0,0)$,
where both its mean and variance are limited for every $(x\neq y$). Hence if
we can prove that $K_{\ell }(x,y;u_{1},u_{2})$ is bounded, by the Lebesgue
Dominated Convergence Theorem we shall have that%
\begin{align*}
& \lim_{u_{1},u_{2}\rightarrow 0}\mathbb{E}\left[ \left\{ \mathcal{N}_{\ell
}(u_{1},u_{2};I)\right\} ^{2}\right] -\mathbb{E}\left[ \left\{ \mathcal{N}%
_{\ell }(0,0;I)\right\} ^{2}\right] \\
& =\lim_{u_{1},u_{2}\rightarrow 0}\int_{\mathbb{S}^{2}\times \mathbb{S}%
^{2}}\left\{ K_{\ell }(x,y;u_{1},u_{2})-K_{\ell }(x,y;0,0)\right\} dxdy=0,
\end{align*}%
whence continuity follows. To prove that $K_{\ell }(x,y;u_{1},u_{2})$ is
bounded, we just need to generalize slightly Lemmas 4.5 and 4.6 from \cite%
{CMW}; in particular, we have that
\begin{equation}
p_{(\nabla f_{\ell }(x),\nabla f_{\ell }(y))}(u_{1},u_{2})\leq \left\{ \det
(A_{\ell }(x,y))\right\} ^{-1/2}\leq \frac{C}{d_{\mathbb{S}^{2}}^{2}(x,y)},%
\text{ for some }C>0,\text{ for all }(x,y)\in \mathbb{S}^{2}\times \mathbb{S}%
^{2};  \label{lemma4.5}
\end{equation}%
in fact, this result was shown to hold in Lemma 4.5 of \cite{CMW}, but only
for $d_{\mathbb{S}^{2}}(x,y)<c,$ some $c>0;$ however for $d_{\mathbb{S}%
^{2}}(x,y)\geq c$ we can use Proposition \ref{Thecondition} to conclude that
the determinant admits a non-zero minimum, because it is a strictly positive
polynomial function on a compact set. Likewise, to bound
\begin{equation*}
\mathbb{E}\left[ \left. \left\vert \det \left( \nabla ^{2}f_{\ell
}(x)\right) \right\vert \left\vert \det \left( \nabla ^{2}f_{\ell
}(y)\right) \right\vert \right\vert \nabla f_{\ell }(x)=\nabla f_{\ell
}(y)=(u_{1},u_{2})\right]
\end{equation*}%
we can argue exactly as in the proof of Lemma 4.6 in \cite{CMW}, the only
difference being that the elements that appear in the Hessian matrices $%
\left\{ \nabla ^{2}f_{\ell }(x)\right\} ,\left\{ \nabla ^{2}f_{\ell
}(y)\right\} $ have non-zero means; more precisely, let us write $%
(X_{1},X_{2},X_{3},Y_{1},Y_{2},Y_{3})$ for the six-dimensional vectors of
second-order derivatives conditioned on $\nabla f_{\ell }(x)=\nabla f_{\ell
}(y)=(u_{1},u_{2})$: it is a Gaussian vector with covariance matrix and
expected value given respectively by
\begin{align*}
\Delta _{\ell }(x,y)& :=C_{\ell }(x,y)-B_{\ell }^{T}(x,y)A_{\ell
}^{-1}(x,y)B_{\ell }(x,y), \\
\mu _{\ell }(x,y)& =\left(
\begin{array}{c}
\mu _{1\ell }(x,y) \\
... \\
\mu _{6\ell }(x,y)%
\end{array}%
\right) :=B_{\ell }^{T}(x,y)A_{\ell }^{-1}(x,y)\left(
\begin{array}{c}
u_{1} \\
u_{2} \\
u_{1} \\
u_{2}%
\end{array}%
\right) .
\end{align*}%
Introduce also the centred six-dimensional vector%
\begin{equation*}
(\widetilde{U}_{1},\widetilde{U}_{2},\widetilde{U}_{3},\widetilde{U}_{4},%
\widetilde{U}_{5},\widetilde{U}%
_{6})^{T}=(U_{1},U_{2},U_{3},U_{4},U_{5},U_{6})-\mu _{\ell }(x,y)\text{ };
\end{equation*}%
Note that the elements of the vector $(\widetilde{U}_{1},\widetilde{U}_{2},%
\widetilde{U}_{3},\widetilde{U}_{4},\widetilde{U}_{5},\widetilde{U}_{6})$
can be interpreted as the second-order derivatives conditioned on the
gradient being equal to zero. We show in Lemma \ref{expectedvalue} in
Appendix C that for any fixed $\ell $%
\begin{equation*}
\mu _{k\ell }(x,y)=O(d_{\mathbb{S}^{2}}(x,y))\text{ , }k=1,...,6,
\end{equation*}%
and it is then readily verified that%
\begin{eqnarray*}
&&\mathbb{E}\left[ \left. \left\vert \det \left( \nabla ^{2}f_{\ell
}(x)\right) \right\vert \left\vert \det \left( \nabla ^{2}f_{\ell
}(y)\right) \right\vert \right\vert \nabla f_{\ell }(x)=\nabla f_{\ell
}(y)=(u_{1},u_{2})\right] \\
&=&\mathbb{E}\left[ \left\vert \det \left(
\begin{array}{cc}
\widetilde{U}_{1}+\mu _{1\ell }(x,y) & \widetilde{U}_{2}+\mu _{2\ell }(x,y)
\\
\widetilde{U}_{2}+\mu _{2\ell }(x,y) & \widetilde{U}_{3}+\mu _{3\ell }(x,y)%
\end{array}%
\right) \right\vert \left\vert \det \left(
\begin{array}{cc}
\widetilde{U}_{4}+\mu _{4\ell }(x,y) & \widetilde{U}_{5}+\mu _{5\ell }(x,y)
\\
\widetilde{U}_{5}+\mu _{5\ell }(x,y) & \widetilde{U}_{6}+\mu _{6\ell }(x,y)%
\end{array}%
\right) \right\vert \right] \\
&=&\mathbb{E}\left[ \left\vert \widetilde{U}_{1}\widetilde{U}_{3}\widetilde{U%
}_{4}\widetilde{U}_{6}\right\vert +\left\vert \widetilde{U}_{1}\widetilde{U}%
_{3}\widetilde{U}_{5}^{2}\right\vert +\left\vert \widetilde{U}_{2}^{2}%
\widetilde{U}_{4}\widetilde{U}_{6}\right\vert +\left\vert \widetilde{U}%
_{2}^{2}\widetilde{U}_{5}^{2}\right\vert \right] +O(d_{\mathbb{S}%
^{2}}^{2}(x,y)) \\
&=&O(d_{\mathbb{S}^{2}}^{2}(x,y))\text{ ,}
\end{eqnarray*}%
where for the last two steps we have used Lemma 4.5 in \cite{CMW}, that
covers the behaviour of the moments for second order derivatives conditioned
on $\nabla f_{\ell }(x)=\nabla f_{\ell }(y)=(0,0)$. Hence for all $(x,y)\in
\mathbb{S}^{2}\times \mathbb{S}^{2}$, there exists some $C>0$, such that%
\begin{equation}
\mathbb{E}\left[ \left. \left\vert \det \left( \nabla ^{2}f_{\ell
}(x)\right) \right\vert \left\vert \det \left( \nabla ^{2}f_{\ell
}(y)\right) \right\vert \right\vert \nabla f_{\ell }(x)=\nabla f_{\ell
}(y)=(u_{1},u_{2})\right] \leq Cd_{\mathbb{S}^{2}}^{2}(x,y)\text{ , }
\label{lemma4.6}
\end{equation}%
and combining together \eqref{lemma4.5} and \eqref{lemma4.6}, we have that
\begin{equation*}
K_{\ell }(x,y;u_{1},u_{2})\leq const\text{ for all }(x,y)\in \mathbb{S}%
^{2}\times \mathbb{S}^{2},
\end{equation*}%
whence the result is established.
\end{proof}

\section{Proof of Theorem \protect\ref{th1} \label{Sec:Proof}}

\subsection{Cholesky decomposition}

Let us now write $\sigma _{\ell }(x)$ for the $5\times 5$ covariance matrix
of the Gaussian random vector $(\nabla f_{\ell }(x),vec(\nabla ^{2}f_{\ell
}(x)))$, i.e. the $5\times 1$ vector that includes the gradient and the
Hessian components of interest. We evaluate the covariance matrix $\sigma
_{\ell }(x)$ at any point $x\in \mathbb{S}^{2}$, and we write it in the
partitioned form
\begin{equation*}
\sigma _{\ell }(x)_{5\times 5}=\left(
\begin{array}{cc}
a_{\ell }(x) & b_{\ell }(x) \\
b_{\ell }^{T}(x) & c_{\ell }(x)%
\end{array}%
\right) ,
\end{equation*}%
where as before the superscript $T$ denotes transposition, and
\begin{equation*}
a_{\ell }(x)=\left(
\begin{array}{cc}
\frac{\lambda _{\ell }}{2} & 0 \\
0 & \frac{\lambda _{\ell }}{2}%
\end{array}%
\right) ,\hspace{1cm}b_{\ell }(x)=\left(
\begin{array}{ccc}
0 & 0 & 0 \\
0 & 0 & 0%
\end{array}%
\right) ,
\end{equation*}%
\begin{equation*}
c_{\ell }(x)=\frac{\lambda _{\ell }^{2}}{8}\left(
\begin{array}{ccc}
3-\frac{2}{\lambda _{\ell }} & 0 & 1+\frac{2}{\lambda _{\ell }} \\
0 & 1-\frac{2}{\lambda _{\ell }} & 0 \\
1+\frac{2}{\lambda _{\ell }} & 0 & 3-\frac{2}{\lambda _{\ell }}%
\end{array}%
\right) .
\end{equation*}%
\medskip

\noindent We follow here the same argument as in \cite{CM2018}; in
particular, we recall that the \emph{Cholesky decomposition} of a Hermitian
positive-definite matrix $A$ takes the form $A=\Lambda \Lambda ^{T},$ where $%
\Lambda $ is a lower triangular matrix with real and positive diagonal
entries, and $\Lambda^{T}$ denotes the conjugate transpose of $\Lambda $. It
is well-known that every Hermitian positive-definite matrix (and thus also
every real-valued symmetric positive-definite matrix) admits a unique
Cholesky decomposition.

By an explicit computation, it is then possible to show that the Cholesky
decomposition of $\sigma _{\ell }$ takes the form $\sigma _{\ell }=\Lambda
_{\ell }\Lambda _{\ell }^{T}$, where
\begin{equation*}
\Lambda _{\ell }=\left(
\begin{array}{ccccc}
\frac{\sqrt{\lambda }_{\ell }}{\sqrt{2}} & 0 & 0 & 0 & 0 \\
0 & \frac{\sqrt{\lambda }_{\ell }}{\sqrt{2}} & 0 & 0 & 0 \\
0 & 0 & \frac{\sqrt{\lambda _{\ell }}\sqrt{3\lambda _{\ell }-2}}{2\sqrt{2}}
& 0 & 0 \\
0 & 0 & 0 & \frac{\sqrt{\lambda _{\ell }}\sqrt{\lambda _{\ell }-2}}{2\sqrt{2}%
} & 0 \\
0 & 0 & \frac{\sqrt{\lambda _{\ell }}(\lambda _{\ell }+2)}{2\sqrt{2}\sqrt{%
3\lambda _{\ell }-2}} & 0 & \frac{\lambda _{\ell }\sqrt{{\lambda _{\ell }-2}}%
}{\sqrt{3\lambda _{\ell }-2}}%
\end{array}%
\right) =:\left(
\begin{array}{ccccc}
\lambda _{1} & 0 & 0 & 0 & 0 \\
0 & \lambda _{1} & 0 & 0 & 0 \\
0 & 0 & \lambda _{3} & 0 & 0 \\
0 & 0 & 0 & \lambda _{4} & 0 \\
0 & 0 & \lambda _{2} & 0 & \lambda _{5}%
\end{array}%
\right);
\end{equation*}%
in the last expression, for notational simplicity we have omitted the
dependence of the $\lambda _{i}$s on $\ell $. The matrix is block diagonal,
because under isotropy the gradient components are independent from the
Hessian when evaluated at the same point. We can hence define a $5$%
-dimensional standard Gaussian vector $%
Y(x)=(Y_{1}(x),Y_{2}(x),Y_{3}(x),Y_{4}(x),Y_{5}(x))$ with independent
components such that%
\begin{equation*}
(\nabla f_{\ell }(x),vec(\nabla ^{2}f_{\ell }(x)))=\Lambda _{\ell }Y(x)
\end{equation*}%
\begin{equation*}
=\left( \lambda _{1}Y_{1}(x),\lambda _{1}Y_{2}(x),\lambda
_{3}Y_{3}(x),\lambda _{4}Y_{4}(x),\lambda _{5}Y_{5}(x)+\lambda
_{2}Y_{3}(x)\right).
\end{equation*}%
So we have
\begin{align*}
{\mathcal{N}}_{\ell ,\varepsilon }^{c}(I)& =\int_{\mathbb{S}^{2}}|\text{det}%
H_{f_{\ell }}(x)|\mathbb{I}_{\left\{ (\partial _{11,x}f_{\ell }(x)+\partial
_{22,x}f_{\ell }(x))/\lambda _{\ell }\in I\right\} }\delta _{\varepsilon
}(\nabla f_{\ell }(x))dx \\
& =\int_{\mathbb{S}^{2}}|\lambda _{3}Y_{3}(x)(\lambda _{5}Y_{5}(x)+\lambda
_{2}Y_{3}(x))-(\lambda _{4}Y_{4}(x))^{2}|\mathbb{I}_{\{\frac{\lambda
_{2}+\lambda _{3}}{\lambda }Y_{3}+\frac{\lambda _{5}}{\lambda }Y_{5}\in
I\}}\;\delta _{\varepsilon }(\lambda _{1}Y_{1}(x),\lambda _{1}Y_{2}(x))dx \\
& =\frac{\lambda ^{2}}{\lambda _{1}^{2}}\int_{\mathbb{S}^{2}}\left\vert
\frac{\lambda _{3}\lambda _{5}}{\lambda ^{2}}Y_{3}(x)Y_{5}(x)+\frac{\lambda
_{2}\lambda _{3}}{\lambda ^{2}}Y_{3}^{2}(x)-\frac{\lambda _{4}^{2}}{\lambda
^{2}}Y_{4}^{2}(x)\right\vert \mathbb{I}_{\{\frac{\lambda _{2}+\lambda _{3}}{%
\lambda }Y_{3}+\frac{\lambda _{5}}{\lambda }Y_{5}\in I\}}\;\delta
_{\varepsilon }(Y_{1}(x),Y_{2}(x))dx.
\end{align*}

\subsection{Second order chaotic component}

Following the same argument as in \cite{CM2018}, it can be shown that the
second order chaotic component of the number of critical points is given by
\begin{equation*}
{\mathcal{N}}_{\ell ;I}^{c}[2]=\frac{\lambda ^{2}}{\lambda _{1}^{2}}\left[
\sum_{i<j}h_{ij}(\ell ;I)\int_{\mathbb{S}^{2}}Y_{i}(x)Y_{j}(x)dx+\frac{1}{2}%
\sum_{i=1}^{5}k_{i}(\ell ;I)\int_{\mathbb{S}^{2}}H_{2}(Y_{i}(x))dx\right] ,
\end{equation*}%
where
\begin{align}
& h_{ij}(\ell ;I)=\lim_{\varepsilon \rightarrow 0}\mathbb{E}\left[
\left\vert \frac{\lambda _{3}\lambda _{5}}{\lambda ^{2}}Y_{3}Y_{5}+\frac{%
\lambda _{2}\lambda _{3}}{\lambda ^{2}}Y_{3}^{2}-\frac{\lambda _{4}^{2}}{%
\lambda ^{2}}Y_{4}^{2}\right\vert Y_{i}Y_{j}\;\mathbb{I}_{\{\frac{\lambda
_{2}+\lambda _{3}}{\lambda }Y_{3}+\frac{\lambda _{5}}{\lambda }Y_{5}\in
I\}}\delta _{\varepsilon }(Y_{1},Y_{2})\right] ,  \label{abs1} \\
& k_{i}(\ell ;I)=\lim_{\varepsilon \rightarrow 0}\mathbb{E}\left[ \left\vert
\frac{\lambda _{3}\lambda _{5}}{\lambda ^{2}}Y_{3}Y_{5}+\frac{\lambda
_{2}\lambda _{3}}{\lambda ^{2}}Y_{3}^{2}-\frac{\lambda _{4}^{2}}{\lambda ^{2}%
}Y_{4}^{2}\right\vert H_{2}(Y_{i})\;\mathbb{I}_{\{\frac{\lambda _{2}+\lambda
_{3}}{\lambda }Y_{3}+\frac{\lambda _{5}}{\lambda }Y_{5}\in I\}}\delta
_{\varepsilon }(Y_{1},Y_{2})\right] .  \label{abs2}
\end{align}%
Note, however, that the computation of projection coefficients here is
different (and considerably more complicated) than in \cite{CM2018}, due to
the presence of the absolute value in the previous formulae (\ref{abs1},\ref%
{abs2}), which makes the evaluation of exact moments much more challenging.
In particular, the computation of the projection coefficients $%
h_{ij}(.,.),k_{i}(.,.)$ is collected in two Lemmas below.

\begin{lemma}
\label{proj=0}For every $I\subset \mathbb{R}$, it holds that $h_{1j}(\ell
;I),h_{2j}(\ell ;I)=0$, for $j=1,\dots 5$.
\end{lemma}

\begin{proof}
The statement follows immediately from
\begin{equation}
\lim_{\varepsilon \rightarrow 0}\mathbb{E}[H_{1}(Y)\delta _{\varepsilon
}(Y)]=0.  \label{Hp}
\end{equation}
\end{proof}

\begin{lemma}
\label{proj=01} For every $I\subset \mathbb{R}$, it holds that $h_{34}(\ell
;I),h_{45}(\ell ;I)=O(\ell ^{-1})$.
\end{lemma}

\begin{proof}
Note that
\begin{equation*}
\lim_{\varepsilon \rightarrow 0}\mathbb{E}[H_{0}(Y)\delta _{\varepsilon
}(Y)]=\frac{1}{\sqrt{2\pi }}.
\end{equation*}%
To prove that $h_{34}(\ell ;I)=O(\ell ^{-1})$ we first apply (\ref{Hp}) to
obtain
\begin{align*}
h_{34}(\ell ;I)& =\lim_{\varepsilon \rightarrow 0}\mathbb{E}\left[
\left\vert \frac{\lambda _{3}\lambda _{5}}{\lambda ^{2}}Y_{3}Y_{5}+\frac{%
\lambda _{2}\lambda _{3}}{\lambda ^{2}}Y_{3}^{2}-\frac{\lambda _{4}^{2}}{%
\lambda ^{2}}Y_{4}^{2}\right\vert Y_{4}Y_{3}\;\mathbb{I}_{\{\frac{\lambda
_{2}+\lambda _{3}}{\lambda }Y_{3}+\frac{\lambda _{5}}{\lambda }Y_{5}\in I\}}%
\right] \mathbb{E}[\delta _{\varepsilon }(Y_{1},Y_{2})] \\
& =\frac{1}{2\pi }\mathbb{E}\left[ \left\vert \frac{\lambda _{3}\lambda _{5}%
}{\lambda ^{2}}Y_{3}Y_{5}+\frac{\lambda _{2}\lambda _{3}}{\lambda ^{2}}%
Y_{3}^{2}-\frac{\lambda _{4}^{2}}{\lambda ^{2}}Y_{4}^{2}\right\vert
Y_{4}Y_{3}\;\mathbb{I}_{\{\frac{\lambda _{2}+\lambda _{3}}{\lambda }Y_{3}+%
\frac{\lambda _{5}}{\lambda }Y_{5}\in I\}}\right] \\
& =\frac{1}{2\pi }\mathbb{E}\left[ \left\vert \frac{1}{\sqrt{8}}Y_{3}Y_{5}+%
\frac{1}{8}Y_{3}^{2}-\frac{1}{8}Y_{4}^{2}\right\vert Y_{4}Y_{3}\;\mathbb{I}%
_{\{\frac{\sqrt{2}}{\sqrt{3}}Y_{3}+\frac{1}{\sqrt{3}}Y_{5}\in I\}}\right]
+O(\ell ^{-1}),
\end{align*}%
then we use the following transformation
\begin{equation*}
Y_{3}=\frac{1}{\sqrt{3}}Z_{1},\hspace{0.5cm}Y_{4}=Z_{2},\hspace{0.5cm}Y_{5}=%
\frac{\sqrt{3}}{\sqrt{8}}Z_{3}-\frac{1}{\sqrt{3}\sqrt{8}}Z_{1}
\end{equation*}%
where $Z=(Z_{1},Z_{2},Z_{3})$ is a centred jointly Gaussian random vector
with covariance matrix
\begin{equation*}
\left(
\begin{array}{ccc}
3 & 0 & 1 \\
0 & 1 & 0 \\
1 & 0 & 3%
\end{array}%
\right) .
\end{equation*}%
We can rewrite $h_{34}(\ell ;I)$ as follows
\begin{align*}
h_{34}(\ell ;I)& =\frac{1}{2\pi }\frac{1}{8\sqrt{3}}\mathbb{E}\left[
\left\vert Z_{1}Z_{3}-Z_{2}^{2}\right\vert Z_{1}Z_{2}\;\mathbb{I}_{\{\frac{1%
}{\sqrt{8}}(Z_{1}+Z_{3})\in I\}}\right] +O(\ell ^{-1}) \\
& =\frac{1}{2\pi }\frac{1}{8\sqrt{3}}\int_{I}dt\;\mathbb{E}\left[ \left\vert
Z_{1}Z_{3}-Z_{2}^{2}\right\vert Z_{1}Z_{2}\;\mathbb{I}_{\{Z_{1}+Z_{3}=t\sqrt{%
8}\}}\right] +O(\ell ^{-1}).
\end{align*}%
We introduce the transformation $W_{1}=Z_{1}$, $W_{2}=Z_{2}$, $%
W_{3}=Z_{1}+Z_{3}$ so that
\begin{align*}
\mathbb{E}\left[ \left\vert Z_{1}Z_{3}-Z_{2}^{2}\right\vert Z_{1}Z_{2}\;%
\mathbb{I}_{\{Z_{1}+Z_{3}=t\sqrt{8}\}}\right] & =\mathbb{E}\left[ \left\vert
W_{1}(W_{3}-W_{1})-W_{2}^{2}\right\vert W_{1}W_{2}\;\mathbb{I}_{\{W_{3}=t%
\sqrt{8}\}}\right] \\
& =\phi _{W_{3}}(t\sqrt{8})\mathbb{E}\left[ \left\vert
W_{1}(W_{3}-W_{1})-W_{2}^{2}\right\vert W_{1}W_{2}\Big|W_{3}=t\sqrt{8}\right]
,
\end{align*}%
where $\phi _{W_{3}}$ is the density function of $W_{3}$. In \cite{CMW} we
proved that the bivariate conditioned Gaussian vector $(W_{1},W_{2})|W_{3}=%
\sqrt{8}t$ is distributed as $(X_{1}+\sqrt{2}t,X_{2})$ where $(X_{1},X_{2})$
are independent standard Gaussian, and hence
\begin{equation*}
\mathbb{E}\left[ \left\vert W_{1}(W_{3}-W_{1})-W_{2}^{2}\right\vert
W_{1}W_{2}\Big|W_{3}=t\sqrt{8}\right] =\mathbb{E}\left[ \left\vert
2t^{2}-X_{1}^{2}-X_{2}^{2}\right\vert (X_{1}+\sqrt{2}t)X_{2}\right] .
\end{equation*}%
We finally note that by symmetry
\begin{equation*}
\mathbb{E}\left[ \left\vert 2t^{2}-X_{1}^{2}-X_{2}^{2}\right\vert X_{2}%
\right] =0\;\text{and}\;\mathbb{E}\left[ \left\vert
2t^{2}-X_{1}^{2}-X_{2}^{2}\right\vert X_{1}X_{2}\right] =0.
\end{equation*}%
In the same way we see that $h_{45}(\ell ;I)=O(\ell ^{-1})$.
\end{proof}

\noindent For convenience, we recall now in the following Lemma the results
in \cite{CM2018}, Proposition 6:

\begin{lemma}
\label{AB} We have
\begin{equation*}
\int_{\mathbb{S}^{2}}Y_{3}(x)Y_{5}(x)dx=4\pi \frac{\sqrt{2}}{3}\sum_{m=-\ell
}^{\ell }\{|a_{\ell m}|^{2}-1\}\left[ -\frac{1}{\ell }+\frac{3\,m}{\ell ^{2}}%
-\frac{2\,m^{3}}{\ell ^{4}}\right] +r_{0}(\ell ),
\end{equation*}
and moreover
\begin{align*}
\int_{\mathbb{S}^{2}}H_{2}(Y_{1}(x))dx&=4\pi \,\sum_{m=-\ell }^{\ell
}\{|a_{\ell m}|^{2}-1\}\left[ \frac{1}{\ell }-\frac{m}{\ell ^{2}}\right]
+r_{1}(\ell ), \\
\int_{\mathbb{S}^{2}}H_{2}(Y_{2}(x))dx &=4\pi \,\sum_{m=-\ell }^{\ell
}\{|a_{\ell m}|^{2}-1\}\frac{m}{\ell ^{2}} +r_{2}(\ell ), \\
\int_{\mathbb{S}^{2}}H_{2}(Y_{3}(x))dx&=4\pi \sum_{m=-\ell }^{\ell
}\{|a_{\ell m}|^{2}-1\}\left[ \frac{4}{3\ell }-\frac{2m}{\ell ^{2}}+\frac{%
2m^{3}}{3\ell ^{4}}\right] +r_{3}(\ell ), \\
\int_{\mathbb{S}^{2}}H_{2}(Y_{4}(x))dx&=4\pi \times 2\sum_{m=-\ell }^{\ell
}\{|a_{\ell m}|^{2}-1\}\left[ \frac{m}{\ell ^{2}}-\frac{m^{3}}{\ell ^{4}}%
\right] +r_{4}(\ell ), \\
\int_{\mathbb{S}^{2}}H_{2}(Y_{5}(x))dx&=4\pi \times \frac{1}{6}\sum_{m=-\ell
}^{\ell }\{|a_{\ell m}|^{2}-1\}\left[ \frac{1}{\ell }+\frac{8m^{3}}{\ell ^{4}%
}\right] +r_{5}(\ell ),
\end{align*}
where $\sqrt{\mathbb{E}\left[ r_{i}(\ell )\right] ^{2}}=O(\ell ^{-1})$, for
all $i=0,...,5$.
\end{lemma}

\noindent From Lemma \ref{AB} we deduce the following:

\begin{lemma}
\label{A34A45} We have that
\begin{equation*}
\int_{\mathbb{S}^{2}}Y_{3}(x)Y_{4}(x)dx=r_{6}(\ell ),\hspace{2cm}\int_{%
\mathbb{S}^{2}}Y_{4}(x)Y_{5}(x)dx=r_{7}(\ell ),
\end{equation*}%
where $\sqrt{\mathbb{E}[r_{i}(\ell )]^{2}}=O(1)$, for $i=6,7$.
\end{lemma}

\begin{proof}
Let, for $i,j,k,l=1,2$,
\begin{equation*}
I_{ij,kl}(\ell )=\int_{\mathbb{S}^{2}}\partial _{ij,x}f_{\ell }(x)\partial
_{kl,x}f_{\ell }(x)dx,
\end{equation*}%
and
\begin{equation*}
I_{00}(\ell )=\int_{\mathbb{S}^{2}}f_{\ell }^{2}(x)dx,\hspace{1cm}%
I_{0,22}(\ell )=\int_{\mathbb{S}^{2}}f_{\ell }(x)\partial _{22,x}f_{\ell
}(x)dx.
\end{equation*}%
We immediately see that
\begin{equation*}
\int_{\mathbb{S}^{2}}Y_{3}(x)Y_{4}(x)dx=\frac{1}{\lambda _{3}\lambda _{4}}%
\int_{\mathbb{S}^{2}}\partial _{11,x}f_{\ell }(x)\,\partial _{21,x}f_{\ell
}(x)\,dx=\frac{1}{\lambda _{3}\lambda _{4}}I_{11,12}(\ell ),
\end{equation*}%
and
\begin{align*}
\int_{\mathbb{S}^{2}}Y_{4}(x)Y_{5}(x)dx& =\frac{1}{\lambda _{4}\lambda _{5}}%
\int_{\mathbb{S}^{2}}\partial _{12,x}f_{\ell }(x)\left( \partial
_{22,x}f_{\ell }(x)-\frac{\lambda _{2}}{\lambda _{3}}\partial _{11,x}f_{\ell
}(x)\right) dx \\
& =\frac{1}{\lambda _{4}\lambda _{5}}I_{12,22}(\ell )-\frac{\lambda _{2}}{%
\lambda _{3}\lambda _{4}\lambda _{5}}I_{11,12}(\ell ).
\end{align*}%
Applying Cauchy-Schwarz we also have
\begin{equation*}
I_{11,12}(\ell )\leq \sqrt{I_{11,11}(\ell )\,I_{12,12}(\ell )},\hspace{1cm}%
I_{12,22}(\ell )\leq \sqrt{I_{12,12}(\ell )\,I_{22,22}(\ell )};
\end{equation*}%
where
\begin{equation*}
I_{11,11}(\ell )=\lambda _{\ell }^{2}I_{00}(\ell )+I_{22,22}(\ell )+2\lambda
_{\ell }I_{0,22}(\ell ).
\end{equation*}%
The results in Lemma \ref{AB}:
\begin{align*}
& I_{00}(\ell )=\frac{1}{2\ell +1}\sum_{m=-\ell }^{\ell }|a_{\ell m}|^{2}, \\
& I_{0,22}(\ell )=-a_{\ell 0}^{2}\frac{\ell }{2\ell +1}+\sum_{m>0}|a_{\ell
m}|^{2}(\frac{1}{2\ell +1}-m), \\
& I_{12,12}(\ell )=\sum_{m=-\ell }^{\ell }|a_{\ell m}|^{2}m\left\{ \frac{%
\lambda _{\ell }-1-m^{2}}{4}\right\} , \\
& I_{22,22}(\ell )=\frac{a_{\ell 0}^{2}}{2}\left( \ell ^{2}-\frac{\ell }{%
2\ell +1}\right) +\frac{1}{2}\sum_{m>0}|a_{\ell m}|^{2}\left\{ -\frac{%
4\lambda _{\ell }}{2\ell +1}+m+\lambda m+m^{3}\right\} .
\end{align*}%
immediately imply that
\begin{equation*}
\int_{\mathbb{S}^{2}}Y_{3}(x)Y_{4}(x)dx=r_{6}(\ell ),\hspace{1cm}\int_{%
\mathbb{S}^{2}}Y_{5}(x)Y_{4}(x)dx=r_{7}(\ell );
\end{equation*}%
where $\sqrt{\mathbb{E}[r_{i}(\ell )]^{2}}=O(1)$, for $i=6,7$.
\end{proof}

\noindent In view of Lemmas \ref{proj=0}, \ref{proj=01}, \ref{A34A45} we
have
\begin{equation}
{\mathcal{N}}_{\ell ;I}^{c}[2]=\frac{\lambda ^{2}}{\lambda _{1}^{2}}\left[
h_{35}(\ell ;I)\int_{\mathbb{S}^{2}}Y_{3}(x)Y_{5}(x)dx+\frac{1}{2}%
\sum_{i=1}^{5}k_{i}(\ell ;I)\int_{\mathbb{S}^{2}}H_{2}(Y_{i}(x))dx\right]
+R_{3}(\ell ),  \label{15.58}
\end{equation}%
where $\frac{\lambda ^{2}}{\lambda _{1}^{2}}=2\lambda _{\ell }$ and $\sqrt{%
\mathbb{E}[R_{3}(\ell )]^{2}}=O(\ell )$ uniformly over $I$. Our next step is
now to rewrite ${\mathcal{N}}_{\ell ;I}^{c}[2]$ rearranging the projection
coefficients to make their computations more manageable. Indeed, recalling
their definitions in \ref{abs1}, \ref{abs2} and exploting the analytic
expressions in Lemma \ref{AB}, some tedious but straightforward algebra
yields that%
\begin{equation*}
{\mathcal{N}}_{\ell ;I}^{c}[2]=2\lambda _{\ell }\frac{2}{\ell }%
\,\sum_{m=-\ell }^{\ell }\{|a_{\ell m}|^{2}-1\}\left[ M_{0}+\frac{m}{\ell }%
M_{1}+\frac{1}{3}\frac{m^{3}}{\ell ^{3}}M_{2}\right] +R_{3}(\ell )\text{ ,}
\end{equation*}%
where
\begin{align*}
{M}_{0}(I)& :=\mathbb{E}\left[ \left\vert \frac{1}{\sqrt{8}}Y_{3}Y_{5}+\frac{%
1}{8}Y_{3}^{2}-\frac{1}{8}Y_{4}^{2}\right\vert \left( -\frac{5}{4}+\frac{2}{3%
}Y_{3}^{2}-\frac{\sqrt{2}}{3}Y_{3}Y_{5}+\frac{1}{12}Y_{5}^{2}\right) \;%
\mathbb{I}_{\{\frac{\sqrt{2}}{\sqrt{3}}Y_{3}+\frac{1}{\sqrt{3}}Y_{5}\in I\}}%
\right] \text{ }, \\
{M}_{1}(I)& :=\mathbb{E}\left[ \left\vert \frac{1}{\sqrt{8}}Y_{3}Y_{5}+\frac{%
1}{8}Y_{3}^{2}-\frac{1}{8}Y_{4}^{2}\right\vert \left( -Y_{3}^{2}+Y_{4}^{2}+%
\sqrt{2}Y_{3}Y_{5}\right) \;\mathbb{I}_{\{\frac{\sqrt{2}}{\sqrt{3}}Y_{3}+%
\frac{1}{\sqrt{3}}Y_{5}\in I\}}\right] \text{ }, \\
{M}_{2}(I)& :=\mathbb{E}\left[ \left\vert \frac{1}{\sqrt{8}}Y_{3}Y_{5}+\frac{%
1}{8}Y_{3}^{2}-\frac{1}{8}Y_{4}^{2}\right\vert \left( Y_{3}^{2}-3Y_{4}^{2}-2%
\sqrt{2}Y_{3}Y_{5}+2Y_{5}^{2}\right) \;\mathbb{I}_{\{\frac{\sqrt{2}}{\sqrt{3}%
}Y_{3}+\frac{1}{\sqrt{3}}Y_{5}\in I\}}\right] \text{ }.
\end{align*}%
\noindent It is then possible to prove that:

\begin{lemma}
\label{M1=0} For all $I\subset \mathbb{R},$ we have $M_{1}(I)=M_{2}(I)=0$
and
\begin{equation*}
M_{0}(I)=\frac{1}{8}\int_{I}p_{3}^{c}(t)dt\text{ ,}
\end{equation*}%
where $p_{3}^{c}(.)$ is defined in (\ref{p3}).
\end{lemma}

\begin{proof}
Note that after the transformation
\begin{equation*}
Y_{3}=\frac{1}{\sqrt{3}}Z_{1},\hspace{0.5cm}Y_{4}=Z_{2},\hspace{0.5cm}Y_{5}=%
\frac{\sqrt{3}}{\sqrt{8}}Z_{3}-\frac{1}{\sqrt{3}\sqrt{8}}Z_{1},
\end{equation*}%
we immediately obtain
\begin{equation*}
M_{0}(I)=\frac{1}{8}\frac{1}{32}\mathbb{E}\left[ \left\vert
Z_{1}Z_{3}-Z_{2}^{2}\right\vert [-40+(-3Z_{1}+Z_{3})^{2}]\;\mathbb{I}_{\{%
\frac{1}{\sqrt{8}}(Z_{1}+Z_{3})\in I\}}\right] ,
\end{equation*}%
and
\begin{equation*}
{M}_{1}(I)=\frac{1}{3}{M}_{2}=\frac{1}{8}\frac{1}{2}\mathbb{E}\left[
\left\vert Z_{1}Z_{3}-Z_{2}^{2}\right\vert [2Z_{2}^{2}+Z_{1}(Z_{3}-Z_{1})]\;%
\mathbb{I}_{\{\frac{1}{\sqrt{8}}(Z_{1}+Z_{3})\in I\}}\right] ,
\end{equation*}%
since
\begin{align*}
& \mathbb{E}\left[ \left\vert Z_{1}Z_{3}-Z_{2}^{2}\right\vert
[4Z_{2}^{2}+2Z_{1}Z_{3}-2Z_{1}^{2}]\;\mathbb{I}_{\{\frac{1}{\sqrt{8}}%
(Z_{1}+Z_{3})\in I\}}\right] \\
& =\mathbb{E}\left[ \left\vert Z_{1}Z_{3}-Z_{2}^{2}\right\vert
[-4Z_{2}^{2}+Z_{1}^{2}+Z_{3}^{2}-2Z_{1}Z_{3}]\;\mathbb{I}_{\{\frac{1}{\sqrt{8%
}}(Z_{1}+Z_{3})\in I\}}\right] .
\end{align*}%
We write $W_{1}=Z_{1}$, $W_{2}=Z_{2}$, $W_{3}=Z_{1}+Z_{3}$, i.e. $%
Z_{3}=W_{3}-W_{1}$ so that
\begin{align*}
M_{1}(I)& =\frac{1}{8}\frac{1}{4}\mathbb{E}\left[ \left\vert
W_{1}(W_{3}-W_{1})-W_{2}^{2}\right\vert
[4W_{2}^{2}+2W_{1}(W_{3}-W_{1})-2W_{1}^{2}]\;\mathbb{I}_{\{\frac{1}{\sqrt{8}}%
W_{3}\in I\}}\right] \\
& =\frac{1}{8}\frac{1}{4}\int_{I}\mathbb{E}\left[ \left\vert
W_{1}W_{3}-W_{1}^{2}-W_{2}^{2}\right\vert
[4W_{2}^{2}+2W_{1}W_{3}-4W_{1}^{2}]|W_{3}=\sqrt{8}t\right] \phi _{W_{3}}(%
\sqrt{8}t)dt
\end{align*}%
where%
\begin{align*}
& \mathbb{E}\left[ \left\vert W_{1}W_{3}-W_{1}^{2}-W_{2}^{2}\right\vert
[4W_{2}^{2}+2W_{1}W_{3}-4W_{1}^{2}]|W_{3}=\sqrt{8}t\right] \\
& =\mathbb{E}\left[ \left\vert \sqrt{8}tW_{1}-W_{1}^{2}-W_{2}^{2}\right\vert
[4W_{2}^{2}+2\sqrt{8}tW_{1}-4W_{1}^{2}]|W_{3}=\sqrt{8}t\right]
\end{align*}%
since we know that $(W_{1},W_{2})|W_{3}=\sqrt{8}t$ is distributed as $(X_{1}+%
\sqrt{2}t,X_{2})$ where $(X_{1},X_{2})$ are independent standard Gaussian,
we have
\begin{align*}
& \mathbb{E}\left[ \left\vert \sqrt{8}t(X_{1}+\sqrt{2}t)-(X_{1}+\sqrt{2}%
t)^{2}-X_{2}^{2}\right\vert [4X_{2}^{2}+2\sqrt{8}t(X_{1}+\sqrt{2}t)-4(X_{1}+%
\sqrt{2}t)^{2}]\right] \\
& =\mathbb{E}\left[ \left\vert -X_{1}^{2}-X_{2}^{2}+2t^{2}\right\vert
[4X_{2}^{2}+2\sqrt{8}t(X_{1}+\sqrt{2}t)-4(X_{1}+\sqrt{2}t)^{2}]\right] \\
& =\mathbb{E}\left[ \left\vert -X_{1}^{2}-X_{2}^{2}+2t^{2}\right\vert
[4X_{2}^{2}-4X_{1}^{2}-4\sqrt{2}tX_{1}]\right] \\
& =-4\sqrt{2}t\mathbb{E}\left[ \left\vert
-X_{1}^{2}-X_{2}^{2}+2t^{2}\right\vert X_{1}\right] =0\text{ ,}
\end{align*}%
so that the first part of the Lemma is proved. Also, we have that
\begin{equation*}
M_{0}(I)=\frac{1}{256}[-40\,{\mathcal{I}}_{I,0}+{\mathcal{I}}_{I,2}]
\end{equation*}%
with
\begin{equation*}
{\mathcal{I}}_{I,0}=\mathbb{E}[|Z_{1}Z_{3}-Z_{2}^{2}|\;\mathbb{I}_{\{\frac{1%
}{\sqrt{8}}(Z_{1}+Z_{3})\in I\}}]\text{ },\;{\mathcal{I}}_{I,2}=\mathbb{E}%
[|Z_{1}Z_{3}-Z_{2}^{2}|(Z_{1}-3Z_{3})^{2}\;\mathbb{I}_{\{\frac{1}{\sqrt{8}}%
(Z_{1}+Z_{3})\in I\}}]\text{ }.
\end{equation*}%
The final expression for ${\mathcal{N}}_{\ell ;I}^{c}[2]$ given in Theorem %
\ref{th1} immediately follows by noting that (using the same notation as in
\cite{CMW})
\begin{equation*}
{\mathcal{I}}_{I,0}=\int_{I}p_{0}^{c}(t)dt\text{ },\hspace{1cm}{\mathcal{I}}%
_{I,2}=8\int_{I}p_{2}^{c}(t)dt\text{ },
\end{equation*}%
where
\begin{equation*}
p_{0}^{c}(t)=\sqrt{\frac{2}{\pi }}[2e^{-t^{2}}+t^{2}-1]e^{-\frac{t^{2}}{2}%
},\;\;\;\;p_{2}^{c}(t)=\sqrt{\frac{2}{\pi }}%
[-4+t^{2}+t^{4}+e^{-t^{2}}2(4+3t^{2})]e^{-\frac{t^{2}}{2}}.
\end{equation*}%
Hence%
\begin{equation*}
M_{0}(I)=\frac{1}{8}\left\{ \int_{I}\left( -\frac{5}{4}p_{0}^{c}(t)+\frac{1}{%
4}p_{2}^{c}(t)\right) dt\right\} =\frac{1}{8}\int_{I}p_{3}^{c}(t)dt\text{ ,}
\end{equation*}%
as claimed.
\end{proof}

We have hence obtained
\begin{align*}
{\mathcal{N}}_{\ell ;I}^{c}[2]& =2\lambda _{\ell }\frac{2}{\ell }%
M_{0}(I)\sum_{m=-\ell }^{\ell }\{|a_{\ell m}|^{2}-1\}+R_{3}(\ell ) \\
& =\ell (\ell +1)\frac{4}{\ell }\frac{1}{8}\left\{
\int_{I}p_{3}^{c}(t)dt\right\} \sum_{m=-\ell }^{\ell }\{|a_{\ell
m}|^{2}-1\}+R_{3}(\ell ), \\
& =\frac{(\ell +1)}{2}\left\{ \int_{I}p_{3}^{c}(t)dt\right\} \sum_{m=-\ell
}^{\ell }\{|a_{\ell m}|^{2}-1\}+R_{3}(\ell )\text{ }.
\end{align*}%
Finally, to complete the proof we need only show that
\begin{equation*}
\lim_{\ell \rightarrow \infty }\frac{Var\left\{ {\mathcal{N}}_{\ell
;I}^{c}\right\} }{Var\left\{ {\mathcal{N}}_{\ell ;I}^{c}[2]\right\} }%
=1+\lim_{\ell \rightarrow \infty }\frac{\sum_{q\geq 3}Var\left\{ {\mathcal{N}%
}_{\ell ;I}^{c}[q]\right\} }{Var\left\{ {\mathcal{N}}_{\ell
;I}^{c}[2]\right\} }=0\text{ ;}
\end{equation*}%
indeed we have that%
\begin{align*}
\text{Var}\left( (\ell +1)\frac{1}{2}\int_{I}p_{3}^{c}(t)dt\sum_{m=-\ell
}^{\ell }\{|a_{\ell m}|^{2}-1\}\right) & =\frac{(\ell +1)^{2}}{4}\left\{
\int_{I}p_{3}^{c}(t)dt\right\} ^{2}\text{Var}\left( \sum_{m=-\ell }^{\ell
}\{|a_{\ell m}|^{2}-1\}\right) \\
& =\frac{(\ell +1)^{2}}{4}\left\{ \int_{I}p_{3}^{c}(t)dt\right\} ^{2}2(2\ell
+1) \\
& =\ell ^{3}\left\{ \int_{I}p_{3}^{c}(t)dt\right\} ^{2}+O(\ell ^{2})\text{ ,}
\end{align*}%
where the last term on the right hand side is the variance for ${\mathcal{N}}%
_{\ell ;I}^{c}$ obtained in \cite{CMW}, as claimed$.$

\section{Appendix A: Levi-Civita Connection and the Hessian}

We start by recalling that in the usual spherical coordinates $(\theta,
\varphi )$ the metric tensor on the tangent plane $T(\mathbb{S}^{2})$ is
given by
\begin{equation*}
g(\theta ,\varphi )=\left[
\begin{matrix}
1 & 0 \\
0 & \sin ^{2}\theta%
\end{matrix}%
\right] .
\end{equation*}%
The computation of the number of critical points by means of the Kac-Rice
formula requires the Hessian of our eigenfunctions; the latter requires the
notion of covariant derivatives (i.e., the Levi-Civita connection) on $%
\mathbb{S}^{2}.$ Recall indeed that the Levi-Civita connection is the only
application $\nabla: T(\mathbb{S}^{2})\times T(\mathbb{S}^{2})\rightarrow T(%
\mathbb{S}^{2})$ such that

\begin{itemize}
\item $\nabla _{X}Y$ is $C^{\infty }$-linear in $X,$ i.e. for $f,g\in
C^{\infty }(\mathbb{S}^{2})$ and $X,Z\in T(\mathbb{S}^{2})$, we have that $%
\nabla _{fX+gZ}Y=f\nabla _{X}Y+g\nabla _{Z}Y$

\item $\nabla _{X}Y$ is $\mathbb{R}$-linear in $Y,$ i.e. for $a,b\in \mathbb{%
R}$ and $Y,Z\in T(\mathbb{S}^{2})$, we have that $\nabla _{X}(aY+bZ)=a\nabla
_{X}Y+b\nabla _{X}Z$

\item Leibnitz rule holds, i.e. $\nabla _{X}fY=(Xf)Y+f\nabla _{X}Y$

\item The operator is compatible with the metric, $Xg(Y,Z)=g(\nabla
_{X}Y,Z)+g(Y,\nabla _{X}Z)$

\item The operator is torsion-free, i.e. $XY-YX=\nabla _{X}Y-\nabla _{Y}X$
\end{itemize}

See e.g. \cite{adlertaylor} Chapter 7 for more discussion and details. In
coordinates, the action of the Levi-Civita connection can be obtained by
means of the so-called Christoffel symbols, see e.g. \cite{chavel} Section
I.1. Given a basis $\left\{ e_{1},e_{2}\right\} $ for the tangent plane $T(%
\mathbb{S}^{2}),$ the Christoffel symbols are defined by%
\begin{equation*}
\nabla _{e_{i}}e_{j}=\Gamma _{ij}^{k}e_{k}+\Gamma _{ij}^{l}e_{l}\text{ }%
,\;i,j,k,l=1,2\text{ }.
\end{equation*}%
For instance, considering as a basis the vectors $\left\{ \frac{\partial }{%
\partial \theta },\frac{\partial }{\partial \varphi }\right\} $ the
Christoffel symbols can be shown to be equal to
\begin{equation*}
\Gamma _{\theta \varphi }^{\theta }=\Gamma _{\theta \theta }^{\theta
}=\Gamma _{\varphi \varphi }^{\varphi }=\Gamma _{\theta \theta }^{\varphi
}=0,\hspace{0.5cm}\Gamma _{\varphi \varphi }^{\theta }=-\sin \theta \cos
\theta \text{ },\hspace{0.5cm}\Gamma _{\varphi \theta }^{\varphi }=\cot
\theta \text{ }.
\end{equation*}%
On the other hand, considering the orthonormal vectors $\frac{\partial }{%
\partial \theta }$ and $\frac{\partial }{\sin \theta \partial \varphi },$
using the linearity properties of the Levi Civita connection and the
Leibnitz rule, we obtain also easily
\begin{equation*}
\nabla _{\frac{\partial }{\partial \theta }}\frac{\partial }{\partial \theta
}=0\text{ },\;\;\;\nabla _{\frac{\partial }{\partial \theta }}\frac{\partial
}{\sin \theta \partial \varphi }=0\text{ },\;\;\;\nabla _{\frac{\partial }{%
\sin \theta \partial \varphi }}\frac{\partial }{\partial \theta }=\cot
\theta \frac{\partial }{\sin \theta \partial \varphi }\text{ },\;\;\;\nabla
_{\frac{\partial }{\sin \theta \partial \varphi }}\frac{\partial }{\sin
\theta \partial \varphi }=-\cot \theta \frac{\partial }{\partial \theta }%
\text{ }.
\end{equation*}%
Note indeed that in this framework the torsion-free property yields (taking $%
X=\frac{\partial }{\partial \theta }$ and $Y=\frac{\partial }{\sin \theta
\partial \varphi })$
\begin{equation*}
XY-YX=\frac{\partial }{\partial \theta }\frac{\partial }{\sin \theta
\partial \varphi }-\frac{\partial }{\sin \theta \partial \varphi }\frac{%
\partial }{\partial \theta }=-\cot \theta \frac{\partial }{\sin \theta
\partial \varphi }=\text{ }\nabla _{\frac{\partial }{\partial \theta }}\frac{%
\partial }{\sin \theta \partial \varphi }-\text{ }\nabla _{\frac{\partial }{%
\sin \theta \partial \varphi }}\frac{\partial }{\partial \theta }=\nabla
_{X}Y-\nabla _{Y}X\text{ }.
\end{equation*}%
The (covariant) Hessian is defined as the bilinear form $(\nabla ^{2}f_{\ell
}):T(\mathbb{S}^{2})\times T(\mathbb{S}^{2})\rightarrow \mathbb{R}$ such that%
\begin{equation*}
(\nabla ^{2}f_{\ell })(X,Y):=XYf_{\ell }-\nabla _{X}Yf_{\ell }\text{ };
\end{equation*}%
in coordinates, the (covariant) Hessian matrix is hence obtained by
replacing the elements of the basis $\partial _{1},\partial _{2}$ in this
expression, leading to%
\begin{eqnarray*}
\nabla ^{2}f_{\ell }(x) &=&\left\{ (\nabla ^{2}f_{\ell })(\partial
_{a},\partial _{a})\right\} _{a,b=1,2}=\left\{ \partial _{a}\partial
_{b}f_{\ell }(x)-\nabla _{\partial _{a}}\partial _{b}f_{\ell }(x)\right\}
_{a,b=1,2} \\
&=&\left(
\begin{matrix}
\frac{\partial ^{2}f_{\ell }(x)}{\partial \theta ^{2}}-\Gamma _{\theta
\theta }^{\theta }\frac{\partial f_{\ell }(x)}{\partial \theta }-\Gamma
_{\theta \theta }^{\varphi }\frac{\partial f_{\ell }(x)}{\partial \varphi }
& \frac{1}{\sin \theta _{x}}[\frac{\partial ^{2}f_{\ell }(x)}{\partial
\theta \partial \varphi }-\Gamma _{\varphi \theta }^{\varphi }\frac{\partial
f_{\ell }(x)}{\partial \varphi }-\Gamma _{\theta \varphi }^{\theta }\frac{%
\partial f_{\ell }(x)}{\partial \theta }] \\
\frac{1}{\sin \theta _{x}}[\frac{\partial ^{2}f_{\ell }(x)}{\partial \theta
\partial \varphi }-\Gamma _{\varphi \theta }^{\varphi }\frac{\partial
f_{\ell }(x)}{\partial \varphi }-\Gamma _{\theta \varphi }^{\theta }\frac{%
\partial f_{\ell }(x)}{\partial \theta }] & \frac{1}{\sin ^{2}\theta _{x}}[%
\frac{\partial ^{2}f_{\ell }(x)}{\partial \varphi ^{2}}-\Gamma _{\varphi
\varphi }^{\varphi }\frac{\partial f_{\ell }(x)}{\partial \varphi }-\Gamma
_{\varphi \varphi }^{\theta }\frac{\partial f_{\ell }(x)}{\partial \theta }]%
\end{matrix}%
\right) \\
&=&\left(
\begin{array}{cc}
\frac{\partial ^{2}f_{\ell }(x)}{\partial \theta ^{2}} & \frac{1}{\sin
\theta }(\frac{\partial ^{2}f_{\ell }(x)}{\partial \theta \partial \varphi }%
-\cot \theta \frac{\partial f_{\ell }(x)}{\partial \varphi }) \\
\frac{1}{\sin \theta }(\frac{\partial ^{2}f_{\ell }(x)}{\partial \theta
\partial \varphi }-\cot \theta \frac{\partial f_{\ell }(x)}{\partial \varphi
}) & \frac{1}{\sin ^{2}\theta }(\frac{\partial ^{2}f_{\ell }(x)}{\partial
\varphi ^{2}}+\sin \theta \cos \theta \frac{\partial f_{\ell }(x)}{\partial
\theta })%
\end{array}%
\right) \text{ }.
\end{eqnarray*}%
It is important to note that, as expected
\begin{equation*}
\text{Tr}(\nabla ^{2}f_{\ell }(x))=\frac{\partial ^{2}f_{\ell }(x)}{\partial
\theta ^{2}}+\frac{1}{\sin ^{2}\theta }(\frac{\partial ^{2}f_{\ell }(x)}{%
\partial \varphi ^{2}}+\sin \theta \cos \theta \frac{\partial f_{\ell }(x)}{%
\partial \theta })=\Delta _{S^{2}}f_{\ell }(x)\text{ ,}
\end{equation*}%
i.e., the trace of the Hessian operator corresponds to the Laplacian.

\section{\protect\bigskip Appendix B: Technical Lemmas}

This section collects the analytic expressions for the derivatives that we
have exploited to prove our results.

\subsection{Covariances of first derivatives}

\begin{lemma}
\label{1-1;2-2;1-2} For all points $x=(\theta _{x},\varphi _{x})\in \mathbb{S%
}^{2}\setminus \{N,S\}$,
\begin{align*}
\left. \mathbb{E}\left[ \partial _{1;x}f_{\ell }(x)\partial _{1;y}f_{\ell
}(y)\right] \right\vert _{x=y}& =P_{\ell }^{\prime }(1)\text{ }, \\
\left. \mathbb{E}\left[ \partial _{1;x}f_{\ell }(x)\partial _{2;y}f_{\ell
}(y)\right] \right\vert _{x=y}& =0\text{ }, \\
\left. \mathbb{E}\left[ \partial _{2;x}f_{\ell }(x)\partial _{2;y}f_{\ell
}(y)\right] \right\vert _{x=y}& =P_{\ell }^{\prime }(1)\text{ }.
\end{align*}
\end{lemma}

\begin{proof}
We have that%
\begin{align*}
& \left. \mathbb{E}\left[ \partial _{1;x}f_{\ell }(x)\partial _{1;y}f_{\ell
}(y)\right] \right\vert _{x=y} \\
& =\left. P_{\ell }^{\prime \prime }(\left\langle x,y\right\rangle )\left\{
-\cos \theta _{x}\sin \theta _{y}+\sin \theta _{x}\cos \theta _{y}\cos
(\varphi _{x}-\varphi _{y})\right\} \left\{ -\sin \theta _{x}\cos \theta
_{y}+\cos \theta _{x}\sin \theta _{y}\cos (\varphi _{x}-\varphi
_{y})\right\} \right\vert _{x=y} \\
& \;+\left. P_{\ell }^{\prime }(\left\langle x,y\right\rangle )\left\{ \sin
\theta _{x}\sin \theta _{y}+\cos \theta _{x}\cos \theta _{y}\cos (\varphi
_{x}-\varphi _{y})\right\} \right\vert _{x=y} \\
& =\left. P_{\ell }^{\prime \prime }(1)\left\{ -\cos \theta _{x}\sin \theta
_{x}+\sin \theta _{x}\cos \theta _{x}\right\} \left\{ -\sin \theta _{x}\cos
\theta _{x}+\cos \theta _{x}\sin \theta _{x}\right\} \right\vert _{x=y} \\
& \;+\left. P_{\ell }^{\prime }(1)\left\{ \sin ^{2}\theta _{x}+\cos
^{2}\theta _{x}\right\} \right\vert _{x=y}=P_{\ell }^{\prime }(1)\text{ ,}
\end{align*}%
moreover
\begin{align*}
& \left. \mathbb{E}\left[ \partial _{1;x}f_{\ell }(x)\partial _{2;y}f_{\ell
}(y)\right] \right\vert _{x=y} \\
& =\frac{1}{\sin \theta _{y}}P_{\ell }^{\prime \prime }(\left\langle
x,y\right\rangle )\left\{ -\sin \theta _{x}\cos \theta _{y}+\cos \theta
_{x}\sin \theta _{y}\cos (\varphi _{x}-\varphi _{y})\right\} \left\{ \sin
\theta _{x}\sin \theta _{y}\sin (\varphi _{x}-\varphi _{y})\right\} |_{x=y}
\\
& \;+\frac{1}{\sin \theta _{y}}P_{\ell }^{\prime }(\left\langle
x,y\right\rangle )\cos \theta _{x}\sin \theta _{y}\sin (\varphi _{x}-\varphi
_{y})|_{x=y} \\
& =\frac{1}{\sin \theta _{x}}P_{\ell }^{\prime \prime }(1)\left\{ -\sin
\theta _{x}\cos \theta _{x}+\cos \theta _{x}\sin \theta _{x}\right\} \sin
^{2}\theta _{x}\sin (\varphi _{x}-\varphi _{x})|_{x=y} \\
& \;+\frac{1}{\sin \theta _{x}}P_{\ell }^{\prime }(\left\langle
x,y\right\rangle )\cos \theta _{x}\sin \theta _{x}\sin (\varphi _{x}-\varphi
_{x})|_{x=y}=0\text{ },
\end{align*}%
and finally
\begin{align*}
& \mathbb{E}\left[ \partial _{2;x}f_{\ell }(x)\partial _{2;y}f_{\ell }(y)%
\right] |_{x=y} \\
& =-\frac{1}{\sin \theta _{y}}\frac{1}{\sin \theta _{x}}P_{\ell }^{\prime
\prime }(\left\langle x,y\right\rangle )\left\{ \sin \theta _{x}\sin \theta
_{y}\sin ^{2}(\varphi _{x}-\varphi _{y})\right\} +P_{\ell }^{\prime
}(\left\langle x,y\right\rangle )\cos (\varphi _{x}-\varphi _{y})|_{x=y} \\
& =-P_{\ell }^{\prime \prime }(1)\sin ^{2}(\varphi _{x}-\varphi
_{x})+P_{\ell }^{\prime }(1)\cos (\varphi _{x}-\varphi _{x})|_{x=y}=P_{\ell
}^{\prime }(1)\text{ }.
\end{align*}
\end{proof}

\subsection{Cross-covariances of first- and second-order derivatives}

\begin{lemma}
\label{11-1;11-2}For all points $x=(\theta _{x},\varphi _{x})\in \mathbb{S}%
^{2}\setminus \{N,S\}$,
\begin{equation*}
\left. \mathbb{E}\left[ \partial _{11;x}f_{\ell }(x)\partial _{1;y}f_{\ell
}(y)\right] \right\vert _{x=y}=\left. \mathbb{E}\left[ \partial
_{11;x}f_{\ell }(x)\partial _{2;y}f_{\ell }(y)\right] \right\vert _{x=y}=0%
\text{ .}
\end{equation*}
\end{lemma}

\begin{proof}
We have that
\begin{align*}
& \left. \mathbb{E}\left[ \partial _{11;x}f_{\ell }(x)\partial _{1;y}f_{\ell
}(y)\right] \right\vert _{x=y} \\
& =\left. \partial _{11;x}\left\{ P_{\ell }^{\prime }(\left\langle
x,y\right\rangle )\left\{ -\cos \theta _{x}\sin \theta _{y}+\sin \theta
_{x}\cos \theta _{y}\cos (\varphi _{x}-\varphi _{y})\right\} \right\}
\right\vert _{x=y} \\
& =\left. \partial _{1;x}\left\{ P_{\ell }^{\prime \prime }(\left\langle
x,y\right\rangle )\left\{ -\cos \theta _{x}\sin \theta _{y}+\sin \theta
_{x}\cos \theta _{y}\cos (\varphi _{x}-\varphi _{y})\right\} ^{2}\right\}
\right\vert _{x=y} \\
& +\left. \partial _{1;x}\left\{ P_{\ell }^{\prime }(\left\langle
x,y\right\rangle )\left\{ \cos \theta _{x}\sin \theta _{y}-\sin \theta
_{x}\cos \theta _{y}\cos (\varphi _{x}-\varphi _{y})\right\} \right\}
\right\vert _{x=y} \\
& =\left. \left\{ P_{\ell }^{\prime \prime \prime }(\left\langle
x,y\right\rangle )\left\{ -\cos \theta _{x}\sin \theta _{y}+\sin \theta
_{x}\cos \theta _{y}\cos (\varphi _{x}-\varphi _{y})\right\} ^{3}\right\}
\right\vert _{x=y} \\
& +\left. 2\left\{ P_{\ell }^{\prime \prime }(\left\langle x,y\right\rangle
)\left\{ -\cos \theta _{x}\sin \theta _{y}+\sin \theta _{x}\cos \theta
_{y}\cos (\varphi _{x}-\varphi _{y})\right\} \left\{ \cos \theta _{x}\sin
\theta _{y}-\sin \theta _{x}\cos \theta _{y}\cos (\varphi _{x}-\varphi
_{y})\right\} \right\} \right\vert _{x=y} \\
& +\left. \left\{ P_{\ell }^{\prime \prime }(\left\langle x,y\right\rangle
)\left\{ \cos \theta _{x}\sin \theta _{y}-\sin \theta _{x}\cos \theta
_{y}\cos (\varphi _{x}-\varphi _{y})\right\} ^{2}\right\} \right\vert _{x=y}
\\
& +\left. \left\{ P_{\ell }^{\prime }(\left\langle x,y\right\rangle )\left\{
-\cos \theta _{x}\sin \theta _{y}+\sin \theta _{x}\cos \theta _{y}\cos
(\varphi _{x}-\varphi _{y})\right\} \right\} \right\vert _{x=y}=0\text{ .}
\end{align*}%
Likewise
\begin{align*}
& \left. \mathbb{E}\left[ \partial _{11;x}f_{\ell }(x)\partial _{2;y}f_{\ell
}(y)\right] \right\vert _{x=y} \\
& =\frac{1}{\sin \theta _{y}}\left. \left\{ P_{\ell }^{\prime \prime \prime
}(\left\langle x,y\right\rangle )\left\{ -\cos \theta _{x}\sin \theta
_{y}+\sin \theta _{x}\sin \theta _{y}\cos (\varphi _{x}-\varphi
_{y})\right\} ^{2}\left\{ \sin \theta _{x}\sin \theta _{y}\sin (\varphi
_{x}-\varphi _{y})\right\} \right\} \right\vert _{x=y} \\
& +\frac{1}{\sin \theta _{y}}\left. \left\{ P_{\ell }^{\prime \prime
}(\left\langle x,y\right\rangle )\left\{ \sin \theta _{x}\sin \theta
_{y}+\cos \theta _{x}\sin \theta _{y}\cos (\varphi _{x}-\varphi
_{y})\right\} \left\{ \sin \theta _{x}\sin \theta _{y}\sin (\varphi
_{x}-\varphi _{y})\right\} \right\} \right\vert _{x=y} \\
& +\frac{1}{\sin \theta _{y}}\left. \left\{ P_{\ell }^{\prime \prime
}(\left\langle x,y\right\rangle )\left\{ -\cos \theta _{x}\sin \theta
_{y}+\sin \theta _{x}\sin \theta _{y}\cos (\varphi _{x}-\varphi
_{y})\right\} \left\{ \cos \theta _{x}\sin \theta _{y}\sin (\varphi
_{x}-\varphi _{y})\right\} \right\} \right\vert _{x=y} \\
& +\frac{1}{\sin \theta _{y}}\left. \left\{ P_{\ell }^{\prime \prime
}(\left\langle x,y\right\rangle )\left\{ \cos \theta _{x}\sin \theta
_{y}\sin (\varphi _{x}-\varphi _{y})\right\} \left\{ -\cos \theta _{x}\sin
\theta _{y}+\sin \theta _{x}\sin \theta _{y}\cos (\varphi _{x}-\varphi
_{y})\right\} \right\} \right\vert _{x=y} \\
& -\frac{1}{\sin \theta _{y}}\left. \left\{ P_{\ell }^{\prime }(\left\langle
x,y\right\rangle )\left\{ \sin \theta _{x}\sin \theta _{y}\sin (\varphi
_{x}-\varphi _{y})\right\} \right\} \right\vert _{x=y}=0\text{ .}
\end{align*}
\end{proof}

\begin{lemma}
\label{22-1;22-2}For all points $x=(\theta _{x},\varphi _{x})\in \mathbb{S}%
^{2}\setminus \{N,S\}$,
\begin{align*}
\left. \mathbb{E}\left[ \partial _{22;x}f_{\ell }(x)\partial _{1;y}f_{\ell
}(y)\right] \right\vert _{x=y} &=-\cot \theta _{x}P_{\ell }^{\prime }(1)%
\text{ ,} \\
\left. \mathbb{E}\left[ \partial _{22;x}f_{\ell }(x)\partial _{2;y}f_{\ell
}(y)\right] \right\vert _{x=y} &=0.
\end{align*}
\end{lemma}

\begin{proof}
We have that
\begin{align*}
& \left. \mathbb{E}\left[ \partial _{22;x}f_{\ell }(x)\partial _{1;y}f_{\ell
}(y)\right] \right\vert _{x=y} \\
& =\left. \partial _{22;x}\left\{ P_{\ell }^{\prime }(\left\langle
x,y\right\rangle )\left\{ -\cos \theta _{x}\sin \theta _{y}+\cos \theta
_{x}\sin \theta _{y}\cos (\varphi _{x}-\varphi _{y})\right\} \right\}
\right\vert _{x=y} \\
& =\frac{1}{\sin ^{2}\theta _{x}}\left. \left\{ P_{\ell }^{\prime \prime
\prime }(\left\langle x,y\right\rangle )\left\{ -\sin \theta _{x}\sin \theta
_{y}\sin (\varphi _{x}-\varphi _{y})\right\} \left\{ -\cos \theta _{x}\sin
\theta _{y}+\cos \theta _{x}\sin \theta _{y}\cos (\varphi _{x}-\varphi
_{y})\right\} ^{2}\right\} \right\vert _{x=y} \\
& +\frac{1}{\sin ^{2}\theta _{x}}\left. \left\{ P_{\ell }^{\prime \prime
}(\left\langle x,y\right\rangle )\left\{ -\sin \theta _{x}\sin \theta
_{y}\cos (\varphi _{x}-\varphi _{y})\right\} \left\{ -\cos \theta _{x}\sin
\theta _{y}+\cos \theta _{x}\sin \theta _{y}\cos (\varphi _{x}-\varphi
_{y})\right\} \right\} \right\vert _{x=y} \\
& +\frac{1}{\sin ^{2}\theta _{x}}\left. \left\{ P_{\ell }^{\prime \prime
}(\left\langle x,y\right\rangle )\left\{ -\sin \theta _{x}\sin \theta
_{y}\sin (\varphi _{x}-\varphi _{y})\right\} \left\{ -\cos \theta _{x}\sin
\theta _{y}\sin (\varphi _{x}-\varphi _{y})\right\} \right\} \right\vert
_{x=y} \\
& +\frac{1}{\sin ^{2}\theta _{x}}\left. \left\{ P_{\ell }^{\prime \prime
}(\left\langle x,y\right\rangle )\left\{ -\sin \theta _{x}\sin \theta
_{y}\sin (\varphi _{x}-\varphi _{y})\right\} \left\{ -\cos \theta _{x}\sin
\theta _{y}\sin (\varphi _{x}-\varphi _{y})\right\} \right\} \right\vert
_{x=y} \\
& +\frac{1}{\sin ^{2}\theta _{x}}\left. \left\{ P_{\ell }^{\prime
}(\left\langle x,y\right\rangle )\left\{ -\cos \theta _{x}\sin \theta
_{y}\cos (\varphi _{x}-\varphi _{y})\right\} \right\} \right\vert
_{x=y}=-\cot \theta _{x}P_{\ell }^{\prime }(1)\text{ .}
\end{align*}%
The proof of the second result is similar:%
\begin{align*}
& \left. \mathbb{E}\left[ \partial _{22;x}f_{\ell }(x)\partial _{2;y}f_{\ell
}(y)\right] \right\vert _{x=y} \\
& =\frac{1}{\sin \theta _{y}}\left. \partial _{22;x}\left\{ P_{\ell
}^{\prime }(\left\langle x,y\right\rangle )\left\{ \sin \theta _{x}\sin
\theta _{y}\sin (\varphi _{x}-\varphi _{y})\right\} \right\} \right\vert
_{x=y} \\
& =\frac{1}{\sin ^{2}\theta _{x}}\frac{1}{\sin \theta _{y}}\left. \left\{
P_{\ell }^{\prime \prime \prime }(\left\langle x,y\right\rangle )\left\{
\sin \theta _{x}\sin \theta _{y}\sin (\varphi _{x}-\varphi _{y})\right\}
^{3}\right\} \right\vert _{x=y} \\
& +\frac{1}{\sin ^{2}\theta _{x}}\frac{1}{\sin \theta _{y}}\left. 2\left\{
P_{\ell }^{\prime \prime }(\left\langle x,y\right\rangle )2\left\{ \sin
\theta _{x}\sin \theta _{y}\sin (\varphi _{x}-\varphi _{y})\right\} \left\{
\sin \theta _{x}\sin \theta _{y}\cos (\varphi _{x}-\varphi _{y})\right\}
\right\} \right\vert _{x=y} \\
& -\frac{1}{\sin ^{2}\theta _{x}}\frac{1}{\sin \theta _{y}}\left. \left\{
P_{\ell }^{\prime \prime }(\left\langle x,y\right\rangle )\left\{ \sin
\theta _{x}\sin \theta _{y}\sin (\varphi _{x}-\varphi _{y})\right\} \left\{
\sin \theta _{x}\sin \theta _{y}\cos (\varphi _{x}-\varphi _{y})\right\}
\right\} \right\vert _{x=y} \\
& -\frac{1}{\sin ^{2}\theta _{x}}\frac{1}{\sin \theta _{y}}\left. \left\{
P_{\ell }^{\prime }(\left\langle x,y\right\rangle )\left\{ \sin \theta
_{x}\sin \theta _{y}\sin (\varphi _{x}-\varphi _{y})\right\} \right\}
\right\vert _{x=y}=0\text{ .}
\end{align*}
\end{proof}

\noindent For the cross-derivatives, we shall need also:

\begin{lemma}
\label{12-1;12-2}For all points $x=(\theta _{x},\varphi _{x})\in \mathbb{S}%
^{2}\setminus \{N,S\}$,
\begin{align*}
\left. \mathbb{E}\left[ \partial _{21;x}f_{\ell }(x)\partial _{1;y}f_{\ell
}(y)\right] \right\vert _{x=y} &=0\text{ ,} \\
\left. \mathbb{E}\left[ \partial _{21;x}f_{\ell }(x)\partial _{2;y}f_{\ell
}(y)\right] \right\vert _{x=y} &=\cot \theta _{x}P_{\ell }^{\prime }(1).
\end{align*}
\end{lemma}

\begin{proof}
We have that
\begin{align*}
& \left. \mathbb{E}\left[ \partial _{12;x}f_{\ell }(x)\partial _{1;y}f_{\ell
}(y)\right] \right\vert _{x=y} \\
& =-\frac{1}{\sin \theta _{x}}\left. \left\{ P_{\ell }^{\prime \prime \prime
}(\left\langle x,y\right\rangle )\left\{ \sin \theta _{x}\sin \theta
_{y}\sin (\varphi _{x}-\varphi _{y})\right\} \left\{ -\cos \theta _{x}\sin
\theta _{y}+\sin \theta _{x}\cos \theta _{y}\cos (\varphi _{x}-\varphi
_{y})\right\} ^{2}\right\} \right\vert _{x=y} \\
& -\frac{1}{\sin \theta _{x}}\left. \left\{ P_{\ell }^{\prime \prime
}(\left\langle x,y\right\rangle )\left\{ \cos \theta _{x}\sin \theta
_{y}\sin (\varphi _{x}-\varphi _{y})\right\} \left\{ -\cos \theta _{x}\sin
\theta _{y}+\sin \theta _{x}\cos \theta _{y}\cos (\varphi _{x}-\varphi
_{y})\right\} \right\} \right\vert _{x=y} \\
& -\frac{1}{\sin \theta _{x}}\left. \left\{ P_{\ell }^{\prime \prime
}(\left\langle x,y\right\rangle )\left\{ \sin \theta _{x}\sin \theta
_{y}\sin (\varphi _{x}-\varphi _{y})\right\} \left\{ \sin \theta _{x}\sin
\theta _{y}+\cos \theta _{x}\cos \theta _{y}\cos (\varphi _{x}-\varphi
_{y})\right\} \right\} \right\vert _{x=y} \\
& -\frac{1}{\sin \theta _{x}}\left. \left\{ P_{\ell }^{\prime \prime
}(\left\langle x,y\right\rangle )\left\{ -\sin \theta _{x}\cos \theta
_{y}+\cos \theta _{x}\sin \theta _{y}\cos (\varphi _{x}-\varphi
_{y})\right\} \left\{ \sin \theta _{x}\cos \theta _{y}\sin (\varphi
_{x}-\varphi _{y})\right\} \right\} \right\vert _{x=y} \\
& -\frac{1}{\sin \theta _{x}}\left. \left\{ P_{\ell }^{\prime }(\left\langle
x,y\right\rangle )\left\{ \cos \theta _{x}\cos \theta _{y}\sin (\varphi
_{x}-\varphi _{y})\right\} \right\} \right\vert _{x=y}=0,
\end{align*}%
and more importantly
\begin{align*}
& \left. \mathbb{E}\left[ \partial _{12;x}f_{\ell }(x)\partial _{2;y}f_{\ell
}(y)\right] \right\vert _{x=y} \\
& =\frac{1}{\sin \theta _{y}}\left. \partial _{12;x}\left\{ P_{\ell
}^{\prime }(\left\langle x,y\right\rangle )\left\{ \sin \theta _{x}\sin
\theta _{y}\sin (\varphi _{x}-\varphi _{y})\right\} \right\} \right\vert
_{x=y} \\
& =-\frac{1}{\sin \theta _{x}}\frac{1}{\sin \theta _{y}}\left. \partial
_{1;x}\left\{ P_{\ell }^{\prime \prime }(\left\langle x,y\right\rangle
)\left\{ \sin \theta _{x}\sin \theta _{y}\sin (\varphi _{x}-\varphi
_{y})\right\} ^{2}\right\} \right\vert _{x=y} \\
& +\frac{1}{\sin \theta _{x}}\frac{1}{\sin \theta _{y}}\left. \partial
_{1;x}\left\{ P_{\ell }^{\prime }(\left\langle x,y\right\rangle )\left\{
\sin \theta _{x}\sin \theta _{y}\cos (\varphi _{x}-\varphi _{y})\right\}
\right\} \right\vert _{x=y} \\
& =-\frac{1}{\sin \theta _{x}}\frac{1}{\sin \theta _{y}}\left. \left\{
P_{\ell }^{\prime \prime \prime }(\left\langle x,y\right\rangle )\left\{
-\sin \theta _{x}\cos \theta _{y}+\sin \theta _{x}\sin \theta _{y}\cos
(\varphi _{x}-\varphi _{y})\right\} \left\{ \sin \theta _{x}\sin \theta
_{y}\sin (\varphi _{x}-\varphi _{y})\right\} ^{2}\right\} \right\vert _{x=y}
\\
& -\frac{1}{\sin \theta _{x}}\frac{1}{\sin \theta _{y}}\left. \left\{
P_{\ell }^{\prime \prime }(\left\langle x,y\right\rangle )2\left\{ \sin
\theta _{x}\sin \theta _{y}\sin (\varphi _{x}-\varphi _{y})\right\} \left\{
\cos \theta _{x}\sin \theta _{y}\sin (\varphi _{x}-\varphi _{y})\right\}
\right\} \right\vert _{x=y} \\
& +\frac{1}{\sin \theta _{x}}\frac{1}{\sin \theta _{y}}\left. \left\{
P_{\ell }^{\prime \prime }(\left\langle x,y\right\rangle )\left\{ -\sin
\theta _{x}\cos \theta _{y}+\cos \theta _{x}\sin \theta _{y}\cos (\varphi
_{x}-\varphi _{y})\right\} \left\{ \sin \theta _{x}\sin \theta _{y}\cos
(\varphi _{x}-\varphi _{y})\right\} \right\} \right\vert _{x=y} \\
& +\frac{1}{\sin \theta _{x}}\frac{1}{\sin \theta _{y}}\left. \left\{
P_{\ell }^{\prime }(\left\langle x,y\right\rangle )\left\{ \cos \theta
_{x}\sin \theta _{y}\cos (\varphi _{x}-\varphi _{y})\right\} \right\}
\right\vert _{x=y}=\cot \theta _{x}P_{\ell }^{\prime }(1).
\end{align*}
\end{proof}

\subsection{Covariances of second-order derivatives}

Let us now consider the second-order derivatives.

\begin{lemma}
\label{11-11}For all points $x=(\theta _{x},\varphi _{x})\in \mathbb{S}%
^{2}\setminus \{N,S\}$,
\begin{equation*}
\left. \mathbb{E}\left[ \partial _{11;x}f_{\ell }(x)\partial _{11;y}f_{\ell
}(y)\right] \right\vert _{x=y}=3P_{\ell }^{\prime \prime }(1)+P_{\ell
}^{\prime }(1).
\end{equation*}
\end{lemma}

\begin{proof}
By an explicit computation of derivatives, we have that
\begin{equation*}
\left. \mathbb{E}\left[ \partial _{11;x}f_{\ell }(x)\partial _{11;y}f_{\ell
}(y)\right] \right\vert _{x=y}=A+B+C+D,
\end{equation*}%
where
\begin{align*}
A& =P_{\ell }^{\prime \prime \prime \prime }(\left\langle x,y\right\rangle
)\left\{ -\sin \theta _{x}\cos \theta _{y}+\cos \theta _{x}\sin \theta
_{y}\cos (\varphi _{x}-\varphi _{x})\right\} ^{2}\left\{ -\cos \theta
_{x}\sin \theta _{y}+\sin \theta _{x}\cos \theta _{y}\cos (\varphi
_{x}-\varphi _{x})\right\} ^{2}|_{x=y} \\
& \;+P_{\ell }^{\prime \prime \prime }(\left\langle x,y\right\rangle
)\left\{ -\cos \theta _{x}\cos \theta _{y}-\sin \theta _{x}\sin \theta
_{y}\cos (\varphi _{x}-\varphi _{x})\right\} \left\{ -\cos \theta _{x}\sin
\theta _{y}+\sin \theta _{x}\cos \theta _{y}\cos (\varphi _{x}-\varphi
_{x})\right\} ^{2}|_{x=y} \\
& \;+P_{\ell }^{\prime \prime \prime }(\left\langle x,y\right\rangle
)\left\{ -\sin \theta _{x}\cos \theta _{y}+\cos \theta _{x}\sin \theta
_{y}\cos (\varphi _{x}-\varphi _{x})\right\} \\
& \;\times 2\left\{ \sin \theta _{x}\sin \theta _{y}+\cos \theta _{x}\cos
\theta _{y}\cos (\varphi _{x}-\varphi _{x})\right\} \left\{ -\cos \theta
_{x}\sin \theta _{y}+\sin \theta _{x}\cos \theta _{y}\cos (\varphi
_{x}-\varphi _{x})\right\} |_{x=y} \\
& =0,
\end{align*}%
\begin{align*}
B& =\left. P_{\ell }^{\prime \prime \prime }(\left\langle x,y\right\rangle
)2\left\{ -\cos \theta _{x}\sin \theta _{y}+\sin \theta _{x}\cos \theta
_{y}\cos (\varphi _{x}-\varphi _{x})\right\} ^{2}\left\{ \sin \theta
_{x}\sin \theta _{y}+\cos \theta _{x}\cos \theta _{y}\cos (\varphi
_{x}-\varphi _{x})\right\} \right\vert _{x=y} \\
& \;+\left. P_{\ell }^{\prime \prime }(\left\langle x,y\right\rangle
)2\left\{ \sin \theta _{x}\sin \theta _{y}+\cos \theta _{x}\cos \theta
_{y}\cos (\varphi _{x}-\varphi _{x})\right\} \left\{ \sin \theta _{x}\sin
\theta _{y}+\cos \theta _{x}\cos \theta _{y}\cos (\varphi _{x}-\varphi
_{x})\right\} \right\vert _{x=y} \\
& \;+\left. P_{\ell }^{\prime \prime }(\left\langle x,y\right\rangle
)2\left\{ -\cos \theta _{x}\sin \theta _{y}+\sin \theta _{x}\cos \theta
_{y}\cos (\varphi _{x}-\varphi _{x})\right\} \left\{ \cos \theta _{x}\sin
\theta _{y}-\sin \theta _{x}\cos \theta _{y}\cos (\varphi _{x}-\varphi
_{x})\right\} \right\vert _{x=y} \\
& =2P_{\ell }^{\prime \prime }(1),
\end{align*}%
\begin{align*}
C& =\left. P_{\ell }^{\prime \prime \prime }(\left\langle x,y\right\rangle
)\left\{ -\sin \theta _{x}\cos \theta _{y}+\cos \theta _{x}\sin \theta
_{y}\cos (\varphi _{x}-\varphi _{x})\right\} ^{2}\left\{ -\cos \theta
_{x}\cos \theta _{y}-\sin \theta _{x}\sin \theta _{y}\cos (\varphi
_{x}-\varphi _{x})\right\} \right\vert _{x=y} \\
& \;+\left. P_{\ell }^{\prime \prime }(\left\langle x,y\right\rangle
)\left\{ -\cos \theta _{x}\cos \theta _{y}-\sin \theta _{x}\sin \theta
_{y}\cos (\varphi _{x}-\varphi _{x})\right\} \left\{ -\cos \theta _{x}\cos
\theta _{y}-\sin \theta _{x}\sin \theta _{y}\cos (\varphi _{x}-\varphi
_{x})\right\} \right\vert _{x=y} \\
& \;+\left. P_{\ell }^{\prime \prime }(\left\langle x,y\right\rangle
)\left\{ -\sin \theta _{x}\cos \theta _{y}+\cos \theta _{x}\sin \theta
_{y}\cos (\varphi _{x}-\varphi _{x})\right\} \left\{ \sin \theta _{x}\cos
\theta _{y}-\cos \theta _{x}\sin \theta _{y}\cos (\varphi _{x}-\varphi
_{x})\right\} \right\vert _{x=y} \\
& =P_{\ell }^{\prime \prime }(1)
\end{align*}%
and finally
\begin{align*}
D& =\left. P_{\ell }^{\prime \prime }(\left\langle x,y\right\rangle )\left\{
-\sin \theta _{x}\cos \theta _{y}+\cos \theta _{x}\sin \theta _{y}\cos
(\varphi _{x}-\varphi _{x})\right\} \left\{ \sin \theta _{x}\cos \theta
_{y}-\cos \theta _{x}\sin \theta _{y}\cos (\varphi _{x}-\varphi
_{x})\right\} \right\vert _{x=y} \\
& \;+\left. P_{\ell }^{\prime }(\left\langle x,y\right\rangle )\left\{ \cos
\theta _{x}\cos \theta _{y}+\sin \theta _{x}\sin \theta _{y}\cos (\varphi
_{x}-\varphi _{x})\right\} \right\vert _{x=y} \\
& =P_{\ell }^{\prime }(1).
\end{align*}
\end{proof}

Our next result is the following:

\begin{lemma}
\label{22-11}For all points $x=(\theta _{x},\varphi _{x})\in \mathbb{S}%
^{2}\setminus \{N,S\}$,
\begin{equation*}
\left. \mathbb{E}\left[ \partial _{22;x}f_{\ell }(x)\partial _{11;y}f_{\ell
}(y)\right] \right\vert _{x=y}=P_{\ell }^{\prime \prime }(1)+P_{\ell
}^{\prime }(1).
\end{equation*}
\end{lemma}

\begin{proof}
Again, evaluation of derivatives gives
\begin{align*}
& \left. \mathbb{E}\left[ \partial _{22;x}f_{\ell }(x)\partial
_{11;y}f_{\ell }(y)\right] \right\vert _{x=y} \\
& =\left. \partial _{2;x}P_{\ell }^{\prime \prime \prime }(\left\langle
x,y\right\rangle )\left\{ -\sin \theta _{y}\sin (\varphi _{x}-\varphi
_{x})\right\} \left\{ -\cos \theta _{x}\sin \theta _{y}+\sin \theta _{x}\cos
\theta _{y}\cos (\varphi _{x}-\varphi _{y})\right\} ^{2}\right\vert _{x=y} \\
& \;+\left. \partial _{2;x}P_{\ell }^{\prime \prime }(\left\langle
x,y\right\rangle )2\left\{ -\cos \theta _{x}\sin \theta _{y}+\sin \theta
_{x}\cos \theta _{y}\cos (\varphi _{x}-\varphi _{y})\right\} \left\{ -\cos
\theta _{y}\sin (\varphi _{x}-\varphi _{y})\right\} \right\vert \\
& \;+\left. \partial _{2;x}P_{\ell }^{\prime \prime }(\left\langle
x,y\right\rangle )\left\{ -\sin \theta _{y}\sin (\varphi _{x}-\varphi
_{y})\right\} \left\{ -\cos \theta _{x}\cos \theta _{y}-\sin \theta _{x}\sin
\theta _{y}\cos (\varphi _{x}-\varphi _{y})\right\} \right\vert _{x=y} \\
& \;+\left. \partial _{2;x}P_{\ell }^{\prime }(\left\langle x,y\right\rangle
)\left\{ \sin \theta _{y}\sin (\varphi _{x}-\varphi _{y})\right\}
\right\vert _{x=y} \\
& =A+B+C+D,
\end{align*}%
where
\begin{align*}
A& =\frac{1}{\sin \theta _{x}}P_{\ell }^{\prime \prime \prime \prime
}(\left\langle x,y\right\rangle )\left\{ -\sin \theta _{y}\sin \theta
_{x}\sin (\varphi _{x}-\varphi _{x})\right\} \\
& \;\times \left\{ -\sin \theta _{y}\sin (\varphi _{x}-\varphi _{x})\right\}
\left\{ -\cos \theta _{x}\sin \theta _{y}+\sin \theta _{x}\cos \theta
_{y}\cos (\varphi _{x}-\varphi _{y})\right\} ^{2}|_{x=y} \\
& \;+\frac{1}{\sin \theta _{x}}\left. P_{\ell }^{\prime \prime \prime
}(\left\langle x,y\right\rangle )\left\{ -\sin \theta _{y}\cos (\varphi
_{x}-\varphi _{x})\right\} \left\{ -\cos \theta _{x}\sin \theta _{y}+\sin
\theta _{x}\cos \theta _{y}\cos (\varphi _{x}-\varphi _{y})\right\}
^{2}\right\vert _{x=y} \\
& \;+\frac{1}{\sin \theta _{x}}\left. P_{\ell }^{\prime \prime \prime
}(\left\langle x,y\right\rangle )\left\{ -\sin \theta _{y}\cos (\varphi
_{x}-\varphi _{x})\right\} \left\{ -\cos \theta _{x}\sin \theta _{y}+\sin
\theta _{x}\cos \theta _{y}\cos (\varphi _{x}-\varphi _{y})\right\}
^{2}\right\vert _{x=y} \\
& =0,
\end{align*}%
\begin{align*}
B& =\left. P_{\ell }^{\prime \prime \prime }(\left\langle x,y\right\rangle
)\left\{ -\sin \theta _{y}\sin (\varphi _{x}-\varphi _{y})\right\} 2\left\{
-\cos \theta _{x}\sin \theta _{y}+\sin \theta _{x}\cos \theta _{y}\cos
(\varphi _{x}-\varphi _{y})\right\} \left\{ -\cos \theta _{y}\sin (\varphi
_{x}-\varphi _{y})\right\} \right\vert _{x=y} \\
& \;+\left. P_{\ell }^{\prime \prime }(\left\langle x,y\right\rangle
)2\left\{ -\cos \theta _{y}\sin (\varphi _{x}-\varphi _{y})\right\} \left\{
-\cos \theta _{y}\sin (\varphi _{x}-\varphi _{y})\right\} \right\vert _{x=y}
\\
& \;+\frac{1}{\sin \theta _{x}}\left. P_{\ell }^{\prime \prime
}(\left\langle x,y\right\rangle )2\left\{ -\cos \theta _{x}\sin \theta
_{y}+\sin \theta _{x}\cos \theta _{y}\cos (\varphi _{x}-\varphi
_{y})\right\} \left\{ -\cos \theta _{y}\cos (\varphi _{x}-\varphi
_{y})\right\} \right\vert _{x=y} \\
& =0,
\end{align*}%
\begin{align*}
C& =P_{\ell }^{\prime \prime \prime }(\left\langle x,y\right\rangle )\left\{
-\sin \theta _{y}\sin (\varphi _{x}-\varphi _{y})\right\} ^{2}\left\{ -\cos
\theta _{x}\cos \theta _{y}-\sin \theta _{x}\sin \theta _{y}\cos (\varphi
_{x}-\varphi _{y})\right\} |_{x=y} \\
& \;+\frac{1}{\sin \theta _{x}}\left. P_{\ell }^{\prime \prime
}(\left\langle x,y\right\rangle )\left\{ -\sin \theta _{y}\cos (\varphi
_{x}-\varphi _{y})\right\} \left\{ -\cos \theta _{x}\cos \theta _{y}-\sin
\theta _{x}\sin \theta _{y}\cos (\varphi _{x}-\varphi _{y})\right\}
\right\vert _{x=y} \\
& \;+\left. \partial _{2;x}P_{\ell }^{\prime \prime }(\left\langle
x,y\right\rangle )\left\{ -\sin \theta _{y}\sin (\varphi _{x}-\varphi
_{y})\right\} \left\{ \sin \theta _{y}\sin (\varphi _{x}-\varphi
_{y})\right\} \right\vert _{x=y} \\
& =P_{\ell }^{\prime \prime }(1),
\end{align*}%
and finally
\begin{align*}
D& =\frac{1}{\sin \theta _{x}}P_{\ell }^{\prime \prime }(\left\langle
x,y\right\rangle )\left\{ -\sin \theta _{x}\sin \theta _{y}\sin (\varphi
_{x}-\varphi _{y})\right\} \left\{ \sin \theta _{y}\sin (\varphi
_{x}-\varphi _{y})\right\} |_{x=y} \\
& \;+\frac{1}{\sin \theta _{x}}P_{\ell }^{\prime }(\left\langle
x,y\right\rangle )\left\{ \sin \theta _{y}\cos (\varphi _{x}-\varphi
_{y})\right\} |_{x=y} \\
& =P_{\ell }^{\prime }(1).
\end{align*}
\end{proof}

\begin{lemma}
\label{21-22}For all points $x=(\theta _{x},\varphi _{x})\in \mathbb{S}%
^{2}\setminus \{N,S\}$,
\begin{equation*}
\left. \mathbb{E}\left[ \partial _{21;x}f_{\ell }(x)\partial _{22;y}f_{\ell
}(y)\right] \right\vert _{x=y}=0.
\end{equation*}
\end{lemma}

\begin{proof}
We have that
\begin{align*}
& \left. \mathbb{E}\left[ \partial _{21;x}f_{\ell }(x)\partial
_{22;y}f_{\ell }(y)\right] \right\vert _{x=y} \\
& =\left. \frac{1}{\sin \theta _{x}}P_{\ell }^{\prime \prime \prime \prime
}(\left\langle x,y\right\rangle )\left\{ -\sin \theta _{x}\cos \theta
_{y}+\cos \theta _{x}\sin \theta _{y}\cos (\varphi _{x}-\varphi
_{y})\right\} ^{2}\left\{ \sin ^{2}\theta _{x}\sin ^{2}(\varphi _{x}-\varphi
_{y})\right\} \right\vert _{x=y} \\
& +\left. \frac{1}{\sin \theta _{x}}P_{\ell }^{\prime \prime \prime
}(\left\langle x,y\right\rangle )\left\{ -\sin \theta _{x}\cos \theta
_{y}+\cos \theta _{x}\sin \theta _{y}\sin (\varphi _{x}-\varphi
_{y})\right\} \left\{ \sin ^{2}\theta _{x}\sin ^{2}(\varphi _{x}-\varphi
_{y})\right\} \right\vert _{x=y} \\
& +\left. \frac{1}{\sin \theta _{x}}P_{\ell }^{\prime \prime \prime
}(\left\langle x,y\right\rangle )\left\{ -\sin \theta _{x}\cos \theta
_{y}+\cos \theta _{x}\sin \theta _{y}\cos (\varphi _{x}-\varphi
_{y})\right\} \left\{ \sin ^{2}\theta _{x}2\sin (\varphi _{x}-\varphi
_{y})\cos (\varphi _{x}-\varphi _{y})\right\} \right\vert _{x=y} \\
& +\left. \frac{1}{\sin \theta _{x}}P_{\ell }^{\prime \prime \prime
}(\left\langle x,y\right\rangle )\left\{ -\sin \theta _{x}\cos \theta
_{y}+\cos \theta _{x}\sin \theta _{y}\cos (\varphi _{x}-\varphi
_{y})\right\} \left\{ 2\sin \theta _{x}\cos \theta _{x}\sin ^{2}(\varphi
_{x}-\varphi _{y})\right\} \right\vert _{x=y} \\
& +\left. \frac{1}{\sin \theta _{x}}P_{\ell }^{\prime \prime }(\left\langle
x,y\right\rangle )\left\{ 2\sin \theta _{x}\cos \theta _{x}2\sin (\varphi
_{x}-\varphi _{y})\cos (\varphi _{x}-\varphi _{y})\right\} \right\vert _{x=y}
\\
& -\left. \frac{1}{\sin \theta _{x}}\frac{1}{\sin \theta _{y}}P_{\ell
}^{\prime \prime \prime }(\left\langle x,y\right\rangle )\left\{ -\sin
\theta _{x}\cos \theta _{y}+\cos \theta _{x}\sin \theta _{y}\cos (\varphi
_{x}-\varphi _{y})\right\} ^{2}\sin \theta _{x}\cos (\varphi _{x}-\varphi
_{x})\right\vert _{x=y} \\
& -\left. \frac{1}{\sin \theta _{x}}\frac{1}{\sin \theta _{y}}P_{\ell
}^{\prime \prime }(\left\langle x,y\right\rangle )\left\{ -\sin \theta
_{x}\cos \theta _{y}-\cos \theta _{x}\sin \theta _{y}\sin (\varphi
_{x}-\varphi _{y})\right\} \sin \theta _{x}\cos (\varphi _{x}-\varphi
_{x})\right\vert _{x=y} \\
& +\left. \frac{1}{\sin \theta _{x}}\frac{1}{\sin \theta _{y}}P_{\ell
}^{\prime \prime }(\left\langle x,y\right\rangle )\left\{ -\sin \theta
_{x}\cos \theta _{y}+\cos \theta _{x}\sin \theta _{y}\cos (\varphi
_{x}-\varphi _{y})\right\} \sin \theta _{x}\sin (\varphi _{x}-\varphi
_{x})\right\vert _{x=y} \\
& -\left. \frac{1}{\sin \theta _{x}}\frac{1}{\sin \theta _{y}}P_{\ell
}^{\prime \prime }(\left\langle x,y\right\rangle )\left\{ -\sin \theta
_{x}\cos \theta _{y}+\cos \theta _{x}\sin \theta _{y}\cos (\varphi
_{x}-\varphi _{y})\right\} \cos \theta _{x}\cos (\varphi _{x}-\varphi
_{x})\right\vert _{x=y} \\
& +\left. \frac{1}{\sin \theta _{x}}\frac{1}{\sin \theta _{y}}P_{\ell
}^{\prime }(\left\langle x,y\right\rangle )\cos \theta _{x}\sin (\varphi
_{x}-\varphi _{x})\right\vert _{x=y},
\end{align*}%
so that for all $x\in \mathbb{S}^{2}$,
\begin{equation*}
\left. \mathbb{E}\left[ \partial _{21;x}T_{\ell }(x)\partial _{22;y}T_{\ell
}(y)\right] \right\vert _{x=y}=0,
\end{equation*}%
as claimed.
\end{proof}

The next variance is more delicate:

\begin{lemma}
\label{22-22}For all points $x=(\theta _{x},\varphi _{x})\in \mathbb{S}%
^{2}\setminus \{N,S\}$,
\begin{equation*}
\left. \mathbb{E}\left[ \partial _{22;x}f_{\ell }(x)\partial _{22;y}f_{\ell
}(y)\right] \right\vert _{x=y}=3P_{\ell }^{\prime \prime }(1)+P_{\ell
}^{\prime }(1)+\cot ^{2}\theta _{x}P_{\ell }^{\prime }(1)\text{ }.
\end{equation*}
\end{lemma}

\begin{proof}
Here we obtain%
\begin{align*}
& \left. \mathbb{E}\left[ \partial _{22;x}f_{\ell }(x)\partial
_{22;y}f_{\ell }(y)\right] \right\vert _{x=y} \\
& =\partial _{22;x}\frac{1}{\sin ^{2}\theta _{y}}P_{\ell }^{\prime \prime
}(\left\langle x,y\right\rangle )\left\{ \sin ^{2}\theta _{x}\sin ^{2}\theta
_{y}\sin ^{2}(\varphi _{x}-\varphi _{y})\right\} |_{x=y} \\
& \;-\partial _{22;x}\frac{1}{\sin ^{2}\theta _{y}}P_{\ell }^{\prime
}(\left\langle x,y\right\rangle )\sin \theta _{x}\sin \theta _{y}\cos
(\varphi _{x}-\varphi _{y})|_{x=y} \\
& =A+B+C+D,
\end{align*}%
for
\begin{align*}
A& =P_{\ell }^{\prime \prime \prime \prime }(\left\langle x,y\right\rangle
)\left\{ -\sin \theta _{y}\sin (\varphi _{x}-\varphi _{y})\right\}
^{2}\left\{ \sin ^{2}\theta _{x}\sin ^{2}(\varphi _{x}-\varphi _{y})\right\}
|_{x=y} \\
& \;+\frac{1}{\sin \theta _{x}}P_{\ell }^{\prime \prime \prime
}(\left\langle x,y\right\rangle )\left\{ -\sin \theta _{y}\cos (\varphi
_{x}-\varphi _{y})\right\} \left\{ \sin ^{2}\theta _{x}\sin ^{2}(\varphi
_{x}-\varphi _{y})\right\} |_{x=y} \\
& \;+\frac{1}{\sin \theta _{x}}P_{\ell }^{\prime \prime \prime
}(\left\langle x,y\right\rangle )\left\{ -\sin \theta _{y}\sin (\varphi
_{x}-\varphi _{y})\right\} \left\{ 2\sin ^{2}\theta _{x}\sin (\varphi
_{x}-\varphi _{y})\cos (\varphi _{x}-\varphi _{y})\right\} |_{x=y}=0,
\end{align*}%
\begin{align*}
B& =\left. P_{\ell }^{\prime \prime \prime }(\left\langle x,y\right\rangle
)\left\{ -\sin \theta _{y}\sin (\varphi _{x}-\varphi _{y})\right\} \left\{
2\sin \theta _{x}\sin (\varphi _{x}-\varphi _{y})\cos (\varphi _{x}-\varphi
_{y})\right\} \right\vert _{x=y} \\
& \;+\left. P_{\ell }^{\prime \prime }(\left\langle x,y\right\rangle
)\left\{ 2\cos (\varphi _{x}-\varphi _{y})\cos (\varphi _{x}-\varphi
_{y})\right\} \right\vert _{x=y} \\
& \;-\left. P_{\ell }^{\prime \prime }(\left\langle x,y\right\rangle
)\left\{ 2\sin (\varphi _{x}-\varphi _{y})\sin (\varphi _{x}-\varphi
_{y})\right\} \right\vert _{x=y}=2P_{\ell }^{\prime \prime }(1),
\end{align*}%
\begin{align*}
C& =-\frac{1}{\sin \theta _{y}}P_{\ell }^{\prime \prime \prime
}(\left\langle x,y\right\rangle )\left\{ -\sin \theta _{y}\sin (\varphi
_{x}-\varphi _{y})\right\} ^{2}\sin \theta _{x}\cos (\varphi _{x}-\varphi
_{y}) \\
& \;+\frac{1}{\sin \theta _{y}}P_{\ell }^{\prime \prime }(\left\langle
x,y\right\rangle )\left\{ \sin \theta _{y}\cos (\varphi _{x}-\varphi
_{y})\right\} \cos (\varphi _{x}-\varphi _{y})|_{x=y} \\
& \;-\frac{1}{\sin \theta _{y}}P_{\ell }^{\prime \prime }(\left\langle
x,y\right\rangle )\left\{ \sin \theta _{y}\sin (\varphi _{x}-\varphi
_{y})\right\} \sin (\varphi _{x}-\varphi _{y})|_{x=y}=P_{\ell }^{\prime
\prime }(1).
\end{align*}%
\begin{align*}
D& =-\frac{1}{\sin \theta _{y}}P_{\ell }^{\prime \prime }(\left\langle
x,y\right\rangle )\sin \theta _{y}\sin (\varphi _{x}-\varphi _{y})\sin
(\varphi _{x}-\varphi _{y})|_{x=y} \\
& \;+\frac{1}{\sin \theta _{x}}\frac{1}{\sin \theta _{y}}P_{\ell }^{\prime
}(\left\langle x,y\right\rangle )\cos (\varphi _{x}-\varphi _{y})|_{x=y}=%
\frac{1}{\sin ^{2}\theta _{x}}P_{\ell }^{\prime }(1).
\end{align*}%
Summing up, we obtain
\begin{equation*}
\left. \mathbb{E}\left[ \partial _{22;x}T_{\ell }(x)\partial _{22;y}T_{\ell
}(y)\right] \right\vert _{x=y}=3P_{\ell }^{\prime \prime }(1)+\frac{1}{\sin
^{2}\theta _{x}}P_{\ell }^{\prime }(1)=3P_{\ell }^{\prime \prime
}(1)+P_{\ell }^{\prime }(1)+\cot ^{2}\theta _{x}P_{\ell }^{\prime }(1).
\end{equation*}
\end{proof}

\begin{lemma}
\label{11-21}For all points $x=(\theta _{x},\varphi _{x})\in \mathbb{S}%
^{2}\setminus \{N,S\}$,
\begin{equation*}
\left. \mathbb{E}\left[ \partial _{11;x}f_{\ell }(x)\partial _{21;y}f_{\ell
}(y)\right] \right\vert _{x=y}=0.
\end{equation*}
\end{lemma}

\begin{proof}
The result can be established by similar computations to those performed in
the other Lemmas; alternatively, in this case it follows immediately by
noting that $\left\{ \partial _{11}f_{\ell }(x)\right\} ,\left\{ \partial
_{21;y}f_{\ell }(x)\right\} $ are the partial derivatives $\partial
_{1},\partial _{2}$ of the constant variance field $\partial _{1}f_{\ell }$.
\end{proof}

Our final lemma is the following

\begin{lemma}
\label{21-21}For all points $x=(\theta _{x},\varphi _{x})\in \mathbb{S}%
^{2}\setminus \{N,S\}$,
\begin{equation*}
\left. \mathbb{E}\left[ \partial _{21;x}f_{\ell }(x)\partial _{21;y}f_{\ell
}(y)\right] \right\vert _{x=y}=P_{\ell }^{\prime \prime }(1)+P_{\ell
}^{\prime }(1)\frac{\cos ^{2}\theta _{x}}{\sin ^{2}\theta _{x}}.
\end{equation*}
\end{lemma}

\begin{proof}
\begin{align*}
& \left. \mathbb{E}\left[ \partial _{21;x}f_{\ell }(x)\partial
_{21;y}f_{\ell }(y)\right] \right\vert _{x=y} \\
& =\left. \partial _{21;x}\partial _{2;y}\left\{ P_{\ell }^{\prime
}(\left\langle x,y\right\rangle )\left\{ -\cos \theta _{x}\sin \theta
_{y}+\sin \theta _{x}\cos \theta _{y}\cos (\varphi _{x}-\varphi
_{y})\right\} \right\} \right\vert _{x=y} \\
& =\left. \partial _{21;x}\left\{ P_{\ell }^{\prime \prime }(\left\langle
x,y\right\rangle )\left\{ \sin \theta _{x}\sin (\varphi _{x}-\varphi
_{y})\right\} \left\{ -\cos \theta _{x}\sin \theta _{y}+\sin \theta _{x}\cos
\theta _{y}\cos (\varphi _{x}-\varphi _{x})\right\} \right\} \right\vert
_{x=y} \\
& \;+\partial _{21;x}\left\{ P_{\ell }^{\prime }(\left\langle
x,y\right\rangle )\left\{ \frac{\sin \theta _{x}}{\sin \theta _{y}}\cos
\theta _{y}\sin (\varphi _{x}-\varphi _{y})\right\} \right\} |_{x=y} \\
& =A+B+C+D+E,
\end{align*}%
where
\begin{align*}
A& =P_{\ell }^{\prime \prime \prime \prime }(\left\langle x,y\right\rangle
)\left\{ -\sin \theta _{y}\sin (\varphi _{x}-\varphi _{y})\right\} \left\{
-\sin \theta _{x}\cos \theta _{y}+\cos \theta _{x}\sin \theta _{y}\cos
(\varphi _{x}-\varphi _{x})\right\} \\
& \;\times \left\{ \sin \theta _{x}\sin (\varphi _{x}-\varphi _{y})\right\}
\left\{ -\cos \theta _{x}\sin \theta _{y}+\sin \theta _{x}\cos \theta
_{y}\cos (\varphi _{x}-\varphi _{x})\right\} |_{x=y} \\
& \;+\left. \left\{ P_{\ell }^{\prime \prime \prime }(\left\langle
x,y\right\rangle )\left\{ -\cos \theta _{x}\sin \theta _{y}\sin (\varphi
_{x}-\varphi _{x})\right\} \left\{ \sin \theta _{x}\sin (\varphi
_{x}-\varphi _{y})\right\} \left\{ -\cos \theta _{x}\sin \theta _{y}+\sin
\theta _{x}\cos \theta _{y}\cos (\varphi _{x}-\varphi _{x})\right\} \right\}
\right\vert _{x=y} \\
& \;+\{P_{\ell }^{\prime \prime \prime }(\left\langle x,y\right\rangle
)\left\{ -\sin \theta _{x}\cos \theta _{y}+\cos \theta _{x}\sin \theta
_{y}\cos (\varphi _{x}-\varphi _{x})\right\} \left\{ \sin \theta _{x}\cos
(\varphi _{x}-\varphi _{y})\right\} \\
& \;\times \left\{ -\cos \theta _{x}\sin \theta _{y}+\sin \theta _{x}\cos
\theta _{y}\cos (\varphi _{x}-\varphi _{x})\right\} \}|_{x=y} \\
& \;+\{P_{\ell }^{\prime \prime \prime }(\left\langle x,y\right\rangle
)\left\{ -\sin \theta _{x}\cos \theta _{y}+\cos \theta _{x}\sin \theta
_{y}\cos (\varphi _{x}-\varphi _{x})\right\} \left\{ \sin \theta _{x}\sin
(\varphi _{x}-\varphi _{y})\right\} \left\{ -\cos \theta _{y}\sin (\varphi
_{x}-\varphi _{x})\right\} \}|_{x=y}=0,
\end{align*}%
\begin{align*}
B& =\left. \partial _{2;x}\left\{ P_{\ell }^{\prime \prime }(\left\langle
x,y\right\rangle )\left\{ \cos \theta _{x}\sin (\varphi _{x}-\varphi
_{y})\right\} \left\{ -\cos \theta _{x}\sin \theta _{y}+\sin \theta _{x}\cos
\theta _{y}\cos (\varphi _{x}-\varphi _{y})\right\} \right\} \right\vert
_{x=y} \\
& =\left. \left\{ P_{\ell }^{\prime \prime \prime }(\left\langle
x,y\right\rangle )\left\{ -\sin \theta _{y}\sin (\varphi _{x}-\varphi
_{x})\right\} \left\{ \cos \theta _{x}\sin (\varphi _{x}-\varphi
_{y})\right\} \left\{ -\cos \theta _{x}\sin \theta _{y}+\sin \theta _{x}\cos
\theta _{y}\cos (\varphi _{x}-\varphi _{y})\right\} \right\} \right\vert
_{x=y} \\
& \;+\left. \left\{ P_{\ell }^{\prime \prime }(\left\langle x,y\right\rangle
)\left\{ \frac{\cos \theta _{x}}{\sin \theta _{x}}\cos (\varphi _{x}-\varphi
_{y})\right\} \left\{ -\cos \theta _{x}\sin \theta _{y}+\sin \theta _{x}\cos
\theta _{y}\cos (\varphi _{x}-\varphi _{y})\right\} \right\} \right\vert
_{x=y} \\
& \;+\left. \left\{ P_{\ell }^{\prime \prime }(\left\langle x,y\right\rangle
)\left\{ \cos \theta _{x}\sin (\varphi _{x}-\varphi _{y})\right\} \left\{
\cos \theta _{y}\sin (\varphi _{x}-\varphi _{y})\right\} \right\}
\right\vert _{x=y}=0,
\end{align*}%
\begin{align*}
C& =\left. \partial _{2;x}\left\{ P_{\ell }^{\prime \prime }(\left\langle
x,y\right\rangle )\left\{ \sin \theta _{x}\sin (\varphi _{x}-\varphi
_{y})\right\} \left\{ \sin \theta _{x}\sin \theta _{y}+\cos \theta _{x}\cos
\theta _{y}\cos (\varphi _{x}-\varphi _{x})\right\} \right\} \right\vert
_{x=y} \\
& =\left. \left\{ P_{\ell }^{\prime \prime \prime }(\left\langle
x,y\right\rangle )\left\{ -\sin \theta _{y}\sin (\varphi _{x}-\varphi
_{y})\right\} \left\{ \sin \theta _{x}\sin (\varphi _{x}-\varphi
_{y})\right\} \left\{ \sin \theta _{x}\sin \theta _{y}+\cos \theta _{x}\cos
\theta _{y}\cos (\varphi _{x}-\varphi _{x})\right\} \right\} \right\vert
_{x=y} \\
& +\left. P_{\ell }^{\prime \prime }(\left\langle x,y\right\rangle )\left\{
\cos (\varphi _{x}-\varphi _{y})\right\} \left\{ \sin \theta _{x}\sin \theta
_{y}+\cos \theta _{x}\cos \theta _{y}\cos (\varphi _{x}-\varphi
_{x})\right\} \right\vert _{x=y} \\
& +\left. \left\{ P_{\ell }^{\prime \prime }(\left\langle x,y\right\rangle
)\left\{ \sin (\varphi _{x}-\varphi _{y})\right\} \left\{ -\cos \theta
_{x}\cos \theta _{y}\sin (\varphi _{x}-\varphi _{x})\right\} \right\}
\right\vert _{x=y}=P_{\ell }^{\prime \prime }(1),
\end{align*}

\begin{align*}
D& =\left. \partial _{2;x}\left\{ P_{\ell }^{\prime \prime }(\left\langle
x,y\right\rangle )\left\{ \frac{\sin \theta _{x}}{\sin \theta _{y}}\cos
\theta _{y}\sin (\varphi _{x}-\varphi _{y})\right\} \left\{ -\sin \theta
_{x}\cos \theta _{y}+\cos \theta _{x}\sin \theta _{y}\cos (\varphi
_{x}-\varphi _{x})\right\} \right\} \right\vert _{x=y} \\
& =\left. \left\{ P_{\ell }^{\prime \prime \prime }(\left\langle
x,y\right\rangle )\left\{ -\sin \theta _{y}\sin (\varphi _{x}-\varphi
_{x})\right\} \left\{ \frac{\sin \theta _{x}}{\sin \theta _{y}}\cos \theta
_{y}\sin (\varphi _{x}-\varphi _{y})\right\} \left\{ -\sin \theta _{x}\cos
\theta _{y}+\cos \theta _{x}\sin \theta _{y}\cos (\varphi _{x}-\varphi
_{x})\right\} \right\} \right\vert _{x=y} \\
& +\left. \left\{ P_{\ell }^{\prime \prime }(\left\langle x,y\right\rangle
)\left\{ \frac{1}{\sin \theta _{y}}\cos \theta _{y}\cos (\varphi
_{x}-\varphi _{y})\right\} \left\{ -\sin \theta _{x}\cos \theta _{y}+\cos
\theta _{x}\sin \theta _{y}\cos (\varphi _{x}-\varphi _{x})\right\} \right\}
\right\vert _{x=y} \\
& +\left. \left\{ P_{\ell }^{\prime \prime }(\left\langle x,y\right\rangle
)\left\{ \frac{1}{\sin \theta _{y}}\cos \theta _{y}\sin (\varphi
_{x}-\varphi _{y})\right\} \frac{1}{\sin \theta _{x}}\left\{ -\sin \theta
_{x}\cos \theta _{y}+\cos \theta _{x}\sin \theta _{y}\sin (\varphi
_{x}-\varphi _{x})\right\} \right\} \right\vert _{x=y}=0,
\end{align*}%
and finally
\begin{align*}
E& =\left. \partial _{2;x}\left\{ P_{\ell }^{\prime }(\left\langle
x,y\right\rangle )\left\{ \frac{\cos \theta _{x}}{\sin \theta _{y}}\cos
\theta _{y}\sin (\varphi _{x}-\varphi _{y})\right\} \right\} \right\vert
_{x=y} \\
& =\left. \left\{ P_{\ell }^{\prime \prime }(\left\langle x,y\right\rangle
)\left\{ -\sin \theta _{y}\sin (\varphi _{x}-\varphi _{y})\right\} \left\{
\frac{\cos \theta _{x}}{\sin \theta _{y}}\cos \theta _{y}\sin (\varphi
_{x}-\varphi _{y})\right\} \right\} \right\vert _{x=y} \\
& \;+\left. \left\{ P_{\ell }^{\prime }(\left\langle x,y\right\rangle
)\left\{ \frac{\cos \theta _{x}}{\sin \theta _{y}}\frac{\cos \theta _{y}}{%
\sin \theta _{x}}\cos (\varphi _{x}-\varphi _{y})\right\} \right\}
\right\vert _{x=y}=P_{\ell }^{\prime }(1)\frac{\cos ^{2}\theta _{x}}{\sin
^{2}\theta _{x}}.
\end{align*}
\end{proof}

\section{Appendix C: The Conditional Expectation}

\begin{lemma}
\label{expectedvalue} For all $(x,y)\in \mathbb{S}^{2}\times \mathbb{S}^{2},$
$x\neq y,$ we have that
\begin{equation*}
\mu _{\ell }(x,y)=O_{\ell }(d_{\mathbb{S}^{2}}(x,y)).
\end{equation*}
\end{lemma}

\begin{proof}
We need to compute the $6\times 4$ matrix
\begin{equation*}
B_{\ell }^{T}(x,y)A_{\ell }^{-1}(x,y)
\end{equation*}%
where%
\begin{equation*}
A_{\ell }(x,y)=\left(
\begin{array}{cccc}
\frac{\lambda _{\ell }}{2} & 0 & \alpha _{1,\ell }(x,y) & 0 \\
0 & \frac{\lambda _{\ell }}{2} & 0 & \alpha _{2,\ell }(x,y) \\
\alpha _{1,\ell }(x,y) & 0 & \frac{\lambda _{\ell }}{2} & 0 \\
0 & \alpha _{2,\ell }(x,y) & 0 & \frac{\lambda _{\ell }}{2}%
\end{array}%
\right);
\end{equation*}%
Let us start by computing
\begin{equation*}
A_{\ell }^{-1}(x,y)=\left(
\begin{array}{cccc}
2\frac{\lambda _{\ell }}{\lambda _{\ell }^{2}-4\alpha _{1,\ell }^{2}(x,y)} &
0 & -\frac{4}{\lambda _{\ell }^{2}-4\alpha _{1,\ell }^{2}(x,y)}\alpha
_{1,\ell }(x,y) & 0 \\
0 & 2\frac{\lambda _{\ell }}{\lambda _{\ell }^{2}-4\alpha _{2,\ell }^{2}(x,y)%
} & 0 & -\frac{4}{\lambda _{\ell }^{2}-4\alpha _{2,\ell }^{2}}\alpha
_{2,\ell }(x,y) \\
-\frac{4}{\lambda _{\ell }^{2}-4\alpha _{1,\ell }^{2}(x,y)}\alpha _{1,\ell
}(x,y) & 0 & 2\frac{\lambda _{\ell }}{\lambda _{\ell }^{2}-4\alpha _{1,\ell
}^{2}(x,y)} & 0 \\
0 & -\frac{4}{\lambda _{\ell }^{2}-4\alpha _{2,\ell }^{2}(x,y)}\alpha
_{2,\ell }(x,y) & 0 & 2\frac{\lambda _{\ell }}{\lambda _{\ell }^{2}-4\alpha
_{2,\ell }^{2}(x,y)}%
\end{array}%
\right).
\end{equation*}%
Hence, omitting for brevity the arguments $(x,y)$%
\begin{equation*}
B_{\ell }^{T}A_{\ell }^{-1}\left(
\begin{array}{c}
u_{1} \\
u_{2} \\
u_{1} \\
u_{2}%
\end{array}%
\right)
\end{equation*}%
\begin{equation*}
=\left(
\begin{array}{cccc}
0 & 0 & 0 & \beta _{2,\ell} \\
0 & 0 & \beta _{1,\ell} & 0 \\
0 & 0 & 0 & \beta _{3,\ell} \\
-\beta _{2,\ell} & 0 & 0 & 0 \\
0 & \beta _{1,\ell} & 0 & 0 \\
-\beta _{3,\ell} & 0 & 0 & 0%
\end{array}%
\right) \left(
\begin{array}{cccc}
2\frac{\lambda _{\ell }}{\lambda _{\ell }^{2}-4\alpha _{1,\ell }^{2}} & 0 & -%
\frac{4}{\lambda _{\ell }^{2}-4\alpha _{1,\ell }^{2}}\alpha _{1,\ell } & 0
\\
0 & 2\frac{\lambda _{\ell }}{\lambda _{\ell }^{2}-4\alpha _{2,\ell }^{2}} & 0
& -\frac{4}{\lambda _{\ell }^{2}-4\alpha _{2,\ell }^{2}}\alpha _{2,\ell } \\
-\frac{4}{\lambda _{\ell }^{2}-4\alpha _{1,\ell }^{2}}\alpha _{1,\ell } & 0
& 2\frac{\lambda _{\ell }}{\lambda _{\ell }^{2}-4\alpha _{1,\ell }^{2}} & 0
\\
0 & -\frac{4}{\lambda _{\ell }^{2}-4\alpha _{2,\ell }^{2}}\alpha _{2,\ell }
& 0 & 2\frac{\lambda _{\ell }}{\lambda _{\ell }^{2}-4\alpha _{2,\ell }^{2}}%
\end{array}%
\right) \left(
\begin{array}{c}
u_{1} \\
u_{2} \\
u_{1} \\
u_{2}%
\end{array}%
\right)
\end{equation*}%
\begin{equation*}
=\left(
\begin{array}{cccc}
0 & -4\frac{\beta _{2,\ell}}{\lambda _{\ell }^{2}-4\alpha _{2,\ell }^{2}}%
\alpha _{2,\ell } & 0 & 2\beta _{2,\ell}\frac{\lambda _{\ell }}{\lambda
_{\ell }^{2}-4\alpha _{2,\ell }^{2}} \\
-4\frac{\beta _{1,\ell}}{\lambda _{\ell }^{2}-4\alpha _{1,\ell }^{2}}\alpha
_{1,\ell } & 0 & 2\beta _{1,\ell}\frac{\lambda _{\ell }}{\lambda _{\ell
}^{2}-4\alpha _{1,\ell }^{2}} & 0 \\
0 & -4\frac{\beta _{3,\ell}}{\lambda _{\ell }^{2}-4\alpha _{2,\ell }^{2}}%
\alpha _{2,\ell } & 0 & 2\beta _{3,\ell}\frac{\lambda _{\ell }}{\lambda
_{\ell }^{2}-4\alpha _{2,\ell }^{2}} \\
-2\beta _{2,\ell}\frac{\lambda _{\ell }}{\lambda _{\ell }^{2}-4\alpha
_{1,\ell }^{2}} & 0 & 4\frac{\beta _{2,\ell}}{\lambda _{\ell }^{2}-4\alpha
_{1,\ell }^{2}}\alpha _{1,\ell } & 0 \\
0 & 2\beta _{1,\ell}\frac{\lambda _{\ell }}{\lambda _{\ell }^{2}-4\alpha
_{2,\ell }^{2}} & 0 & -4\frac{\beta _{1,\ell}}{\lambda _{\ell }^{2}-4\alpha
_{2,\ell }^{2}}\alpha _{2,\ell } \\
-2\beta _{3,\ell}\frac{\lambda _{\ell }}{\lambda _{\ell }^{2}-4\alpha
_{1,\ell }^{2}} & 0 & 4\frac{\beta _{3,\ell}}{\lambda _{\ell }^{2}-4\alpha
_{1,\ell }^{2}}\alpha _{1,\ell } & 0%
\end{array}%
\right) \allowbreak \left(
\begin{array}{c}
u_{1} \\
u_{2} \\
u_{1} \\
u_{2}%
\end{array}%
\right)
\end{equation*}%
\begin{equation*}
=\left(
\begin{array}{c}
\frac{2\beta _{2,\ell}}{\lambda _{\ell }^{2}-4\alpha _{2,\ell }^{2}}(\lambda
_{\ell }u_{2}-2\alpha _{2,\ell }u_{2}) \\
\frac{2\beta _{1,\ell}}{\lambda _{\ell }^{2}-4\alpha _{1,\ell }^{2}}(\lambda
_{\ell }u_{1}-2\alpha _{2,\ell }u_{1}) \\
\frac{2\beta _{3,\ell}}{\lambda _{\ell }^{2}-4\alpha _{2,\ell }^{2}}(\lambda
_{\ell }u_{2}-2\alpha _{2,\ell }u_{2}) \\
-\frac{2\beta _{2,\ell}}{\lambda _{\ell }^{2}-4\alpha _{1,\ell }^{2}}%
(\lambda _{\ell }u_{1}-2\alpha _{2,\ell }u_{1}) \\
\frac{2\beta _{1,\ell}}{\lambda _{\ell }^{2}-4\alpha _{2,\ell }^{2}}(\lambda
_{\ell }u_{2}-2\alpha _{2,\ell }u_{2}) \\
-\frac{2\beta _{3,\ell}}{\lambda _{\ell }^{2}-4\alpha _{1,\ell }^{2}}%
(\lambda _{\ell }u_{1}-2\alpha _{2,\ell }u_{1})%
\end{array}%
\right) \allowbreak =\left(
\begin{array}{c}
\frac{2\beta _{2,\ell}u_{2}}{\lambda _{\ell }+2\alpha _{2,\ell }} \\
\frac{2\beta _{1,\ell}u_{1}}{\lambda _{\ell }+2\alpha _{1,\ell }} \\
\frac{2\beta _{3,\ell}u_{2}}{\lambda _{\ell }+2\alpha _{2,\ell }} \\
-\frac{2\beta _{2,\ell}u_{1}}{\lambda _{\ell }+2\alpha _{1,\ell }} \\
\frac{2\beta _{1,\ell}u_{2}}{\lambda _{\ell }+2\alpha _{2,\ell }} \\
-\frac{2\beta _{3,\ell}u_{1}}{\lambda _{\ell }+2\alpha _{1,\ell }}%
\end{array}%
\right) \allowbreak =O_{\ell }(\varepsilon \sin \phi ).
\end{equation*}%
Let us now focus without loss of generality on ``equator" points $x=(\frac{%
\pi }{2},\varphi _{x}),$ $y=(\frac{\pi }{2},\varphi _{y}),$ and write $\phi
:=d_{\mathbb{S}^{2}}(x,y)=|\varphi _{x}-\varphi _{y}|.$ It has been shown in
\cite{CMW} that the following expressions hold,%
\begin{equation*}
\beta _{1,\ell }(\phi )=\sin \phi P_{\ell }^{\prime \prime }(\cos \phi ),
\end{equation*}%
\begin{equation*}
\beta _{2,\ell }(\phi )=\sin \phi \cos \phi P_{\ell }^{\prime \prime }(\cos
\phi )+\sin \phi P_{\ell }^{\prime }(\cos \phi ),
\end{equation*}%
\begin{equation*}
\beta _{3,\ell }(\phi )=-\sin ^{3}\phi P_{\ell }^{\prime \prime \prime
}(\cos \phi )+3\sin \phi \cos \phi P_{\ell }^{\prime \prime }(\cos \phi
)+\sin \phi P_{\ell }^{\prime }(\cos \phi ),\newline
\end{equation*}%
whence%
\begin{equation*}
\beta _{1,\ell }(\phi ),\beta _{2,\ell }(\phi ),\beta _{3,\ell }(\phi
)=O_{\ell }(\sin \phi )=O_{\ell }(\phi ), \text{ as }\phi \rightarrow 0.
\end{equation*}%
Recall also that%
\begin{equation*}
A_{\ell }(x,y)=
\end{equation*}%
\begin{equation*}
\left(
\begin{array}{cccc}
P_{\ell }^{\prime }(1) & \ast & \ast & \ast \\
0 & P_{\ell }^{\prime }(1) & \ast & \ast \\
P_{\ell }^{\prime }(\left\langle x,y\right\rangle ) & 0 & P_{\ell }^{\prime
}(1) & \ast \\
0 & -P_{\ell }^{\prime \prime }(\left\langle x,y\right\rangle )\sin
^{2}(\varphi _{x}-\varphi _{y})+P_{\ell }^{\prime }(\left\langle
x,y\right\rangle )\cos (\varphi _{x}-\varphi _{y}) & 0 & P_{\ell }^{\prime
}(1)%
\end{array}%
\right)
\end{equation*}%
whence%
\begin{eqnarray*}
\alpha _{1,\ell }(x,y) &=&P_{\ell }^{\prime }(\left\langle x,y\right\rangle )%
\text{ ,} \\
\alpha _{2,\ell }(x,y) &=&-P_{\ell }^{\prime \prime }(\left\langle
x,y\right\rangle )\sin ^{2}(\varphi _{x}-\varphi _{y})+P_{\ell }^{\prime
}(\left\langle x,y\right\rangle )\cos (\varphi _{x}-\varphi _{y})\text{ , }
\end{eqnarray*}%
and thus on the equator $\varphi _{x}=\varphi _{y}$%
\begin{equation*}
\alpha _{1,\ell }(x,y)=P_{\ell }^{\prime }(\left\langle x,y\right\rangle
)=\alpha _{2,\ell }(x,y)\text{ . }
\end{equation*}%
It hence follows that%
\begin{eqnarray*}
\left\vert \mu _{i,\ell }(x,y)\right\vert &\leq &const_{\ell }\times \left\{
\left\vert u_{1}\right\vert +\left\vert u_{2}\right\vert \right\} \sin \phi
\\
\text{ } &\leq &const_{\ell }\times \left\{ \left\vert u_{1}\right\vert
+\left\vert u_{2}\right\vert \right\} \times d_{\mathbb{S}^{2}}(x,y)\text{,}
\end{eqnarray*}%
where the constant can depend on $\left\{ \lambda _{\ell }+2\alpha _{1,\ell
}\right\} ^{-1},$ which is bounded because it is the inverse of a
non-vanishing polynomial in a compact space. Hence the result is established.
\end{proof}

\textbf{Acknowledgements} We are grateful to Maurizia Rossi and Igor Wigman
for many insightful discussions. DM acknowledges the MIUR Excellence
Department Project awarded to the Department of Mathematics, University of
Rome Tor Vergata, CUP E83C18000100006.

\bigskip

Valentina Cammarota

Dipartimento di Scienze Statistiche

Sapienza Universit\`{a} di Roma

Piazzale Aldo Moro, 5

00185 Roma

email: valentina.cammarota@uniroma1.it

\bigskip

Domenico Marinucci

Dipartimento di Matematica

Universit\`{a} di Roma Tor Vergata

Via della Ricerca Scientifica, 1

00133 Roma

email: marinucc@mat.uniroma2.it

\bigskip

\end{document}